\documentclass[11pt]{amsart}
\usepackage[latin1]{inputenc}
\usepackage[T1]{fontenc}
\usepackage{amsfonts}
\usepackage{amsmath}
\usepackage{amssymb}
\usepackage{graphicx}
\usepackage{pstricks}
\usepackage{pst-grad}
\usepackage{multido}
\usepackage{color}
\usepackage{amsthm}
\usepackage{geometry}
\usepackage[all]{xy}
\usepackage[absolute]{textpos}
\newcommand{\cA}{\mathcal{A}}
\newcommand{\cB}{\mathcal{B}}

\newcommand{\cH}{\mathcal{H}}

\newcommand{\cL}{\mathcal{L}}
\newcommand{\cLR}{\mathcal{LR}}
\newcommand{\cM}{\mathcal{M}}

\newcommand{\cP}{\mathcal{P}}

\newcommand{\cR}{\mathcal{R}}

\newcommand{\cT}{\mathcal{T}}

\newcommand{\cZ}{\mathcal{Z}}

\newcommand{\sg}{\mathcal{h}}
\newcommand{\sd}{\mathcal{i}}

\newcommand{\bA}{\mathbf{A}}

\newcommand{\bP}{\mathbf{P}}
\newcommand{\bPr}{\mathbf{P}_{\mathbf{R}}}

\newcommand{\ba}{\mathbf{a}}

\newcommand{\fB}{\mathfrak{B}}
\newcommand{\fC}{\mathfrak{C}}

\newcommand{\fH}{\mathfrak{H}}

\newcommand{\fb}{\mathfrak{b}}

\newcommand{\nZ}{\mathbb{Z}}
\newcommand{\nR}{\mathbb{R}}
\newcommand{\nN}{\mathbb{N}}

\newcommand{\nC}{\mathbb{C}}

\newcommand{\Ga}{\Gamma}

\newcommand{\tB}{\tilde{B}}
\newcommand{\tC}{\tilde{C}}

\newcommand{\tG}{\tilde{G}}

\newcommand{\tc}{\tilde{c}}

\newcommand{\sqs}{\sqsubset}
\newcommand{\sq}{\sqsubseteq}
\newcommand{\ov}{\overline}
\newcommand{\ind}{\underset}

\newcommand{\ra}{\rightarrow}

\newtheorem{Th}{Theorem}[section]

\newtheorem{Lem}[Th]{Lemma}
\newtheorem{Cor}[Th]{Corollary}
\newtheorem{Prop}[Th]{Proposition}
\newtheorem{Def-Prop}[Th]{Definition-Proposition}

\theoremstyle{definition}
\newtheorem{Def}[Th]{Definition}
\newtheorem{Exa}[Th]{Example}

\theoremstyle{remark}
\newtheorem{Rem}[Th]{Remark}

\definecolor{purple}{rgb}{0.8,0.12,0.8}
\definecolor{orange}{rgb}{1.0,0.7,0.0}
\definecolor{pink}{rgb}{1,0.5,0.8}
\definecolor{blackg}{rgb}{0.1,0.25,0.1}
\definecolor{ForestGreen}{cmyk}{0.91,0,0.88,0.42}
\definecolor{Turquoise}{cmyk}{0.85,0,0.20,0}

\begin{document}


\title[Affine cellularity of affine Hecke algebras]{Affine cellularity of affine Hecke algebras of rank two}
\author{J\'er\'emie Guilhot \and Vanessa Miemietz}
\date{\today}

\address{J\'er\'emie Guilhot: School of Mathematics,
University of East Anglia 
Norwich NR4 7TJ, UK}

\email{j.guilhot@uea.ac.uk}

\address{Vanessa Miemietz: School of Mathematics,\\ 
University of East Anglia Norwich NR4 7TJ, UK}

\email{v.miemietz@uea.ac.uk}

\begin{abstract} We show that affine Hecke algebras of rank two with generic parameters are affine cellular in the sense of Koenig-Xi.
\end{abstract}

\maketitle

\section{Introduction}

In order to approach the fundamental problem of classifying the irreducible representations of a given finite-dimensional algebra, the concept of cellularity, defined by Graham and Lehrer \cite{GL}, has proven extremely useful. A cellular algebra comes, by definition, with a finite chain of ideals, whose subquotients, denoted by cells, decompose as left modules into a direct sum of copies of a certain module, called the cell module. This cell module comes with a bilinear form. Factoring out the radical with respect to this form leads to a simple module or zero, and in this way, one obtains a complete set of isomorphism classes of simple modules for the given cellular algebra. Examples of cellular algebras include many finite-dimensional Hecke algebras \cite{geck5}.

Recently, Koenig and Xi \cite{KX} have generalised this concept to algebras over a principal ideal domain $k$ of not necessarily finite dimension, by introducing the notion of an affine cellular algebra. Keeping the idea of having a filtration by a finite chain of ideals, the ideals are now allowed to be of infinite dimension. Where before a cell was isomorphic to a matrix ring over $k$ with a twisted multiplication, it is now isomorphic to a matrix ring (still of finite rank) over a quotient of a polynomial ring over $k$, with a twisted multiplication. The most important class of examples in \cite{KX} of affine cellular algebras is given by the extended affine Hecke algebras of type $A$.

In this article, we prove the following theorem, providing the first examples of affine cellularity for affine Hecke algebras with unequal parameters.

\begin{Th}\label{mainthmintro}
Let $\cH$ be an affine Hecke algebra of rank two defined over $\nC[v,v^{-1}]$ with generic parameters. Then $\cH$ is affine cellular in the sense of Koenig and Xi.\end{Th}

Theorem \ref{mainthmintro} is proved by explicit construction of the associated twisted matrix rings and the isomorphism between these and the corresponding cells. 

\section{Affine cellular algebras} \label{affcell}

Let $k$ be a principal ideal domain. For a $k$-algebra $A$, a $k$-linear anti-automorphism $i$ of $A$ satisfying $i^2 = id_A$ is called a $k$-involution on $A$. For two $k$-modules $V,W$ denote by $\sigma$ the map $V \otimes_k W \ra W \otimes_k V$ given by $\sigma(v \otimes w) = w \otimes v$. If $B=k[t_1, \dots, t_r]/I$ for some ideal $I$ in a polynomial ring in finitely many variables over $K$, then $B$ is called an affine $k$-algebra. 
\begin{Def}\cite[Definition 2.1]{KX}
Let $A$ be a unitary $k$-algebra with a $k$-involution $i$. A two-sided ideal $J$ in $A$ is called an \emph{affine cell ideal} if and only if the following two conditions are satisfied.
\begin{enumerate}
\item We have $i(J) = J$.
\item There exists an affine $k$ algebra $B$ with a $k$-involution $\nu$, and a free $k$-module $V$ of finite rank such that $\Delta:=V \otimes_k B$ is an $A$-$B$-bimodule, where the right $B$-structure is induced by the regular $B$-module $B_B$.  
\item There is an $A$-$A$-bimodule isomorphism $\alpha: J \ra \Delta \otimes_B\Delta'$, where $\Delta' = B \otimes_k V$ is the $B$-$A$-bimodule with left $B$-structure induced by ${}_BB$ and right $A$ structure defined by $(b\otimes v)a = \sigma(i(a)(v \otimes b))$, such that the following diagram commutes:

$$\xymatrix{ J \ar^{\alpha}[rr]  \ar^{i}[d]&& \Delta \otimes_B\Delta' \ar^{v \otimes b \otimes_B b' \otimes w \mapsto w \otimes \nu(b') \otimes_B \nu(b) \otimes v }[d] \\ J \ar^{\alpha}[rr]&&\Delta \otimes_B\Delta'.}$$
\end{enumerate}

The algebra $A$ together with its $k$-involution $i$ is called \emph{affine cellular} if and only if there is a $k$-module decomposition $A= J_1' \oplus J_2' \oplus \cdots \oplus J_n'$ for some $n$ with $i(J_l')=J_l$ for $1 \leq l \leq n$, such that, setting $J_m:= \bigoplus_{l=1}^m J_l'$, we obtain a filtration
$$0=J_0 \subset J_1 \subset J_2 \subset \cdots \subset J_n=A$$
of $A$ by two-sided ideals, where each $J_m' = J_m/J_{m-1}$ is an affine cell ideal of $A/J_{m-1}$ (with respect to the involution induced by $i$ on the quotient).
\end{Def}
 
For an affine $k$-algebra $B$ with a $k$-involution $\nu$, a free $k$-module $V$ of finite rank and a $k$-bilinear form $\varphi:V\times V \ra B$, denote by $\mathbb{A}(V,B,\varphi)$ the (possibly non-unital) algebra given as a $k$-module by $V\otimes_k B \otimes_k V$, on which we impose the multiplication $(v_1 \otimes b_1\otimes w_1)(v_2 \otimes b_2 \otimes w_2) := v_1 \otimes b_1 \varphi(w_1,v_2) b_2 \otimes w_2$. 

The description of affine cell ideal we are going to use is the following:

\begin{Prop}\cite[Proposition 2.3]{KX} \label{descrip}
Let $k$ be a principal ideal domain, $A$ a unitary $k$-algebra with a $k$-involution $i$. A two-sided ideal $J$ in $A$ is an affine cell ideal if and only if there exists an affine $k$-algebra $B$ with a $k$-involution $\nu$, a free $k$-module $V$ of finite rank and a bilinear form $\varphi: V \otimes V \ra B$, and an $A$-$A$-bimodule structure on $V \otimes_k B\otimes_kV$, such that $J \cong \mathbb{A}(V,B,\varphi)$ as an algebra and an $A$-$A$-bimodule, and such that under this isomorphism the $k$-involution $i$ restricted to $J$ corresponds to the $k$-involution given by $v \otimes b\otimes w \mapsto w \otimes \nu(b) \otimes v$.
\end{Prop}

Let now $A$ be an affine cellular algebra with a cell chain $0=J_0 \subset J_1 \subset J_2 \subset \cdots \subset J_n=A$, such that each subquotient $J_i/J_{i-1}$ is an affine cell ideal in $A/J_{i-1}$. Then  $J_i/J_{i-1}$ is isomorphic to $\mathbb{A}(V_i,B_i,\varphi_i)$ for some finite-dimensional vector space $V_i$, a commutative $k$-algebra $B_i$ and a bilinear form $\varphi_i: V_i \times V_i \ra B_i$. Let $(\phi^i_{st})$ be the matrix representing the bilinear form $\phi$ with respect to some choice of basis of $V_i$.
Then Koenig and Xi obtain a parameterisation of simple modules over an affine cellular algebra by establishing a bijection between isomorphism classes of  simple $A$-modules and the set 
$$\{ (j, \mathrm{m}) \mid 1 \leq j\leq n,\mathrm{m} \in \mathrm{MaxSpec}(B_j) \textrm{ such that some } \phi^j_{st} \not\in \mathrm{m}\}$$
where $\mathrm{MaxSpec}(B_j)$ denotes the maximal ideal spectrum of $B_j$.
Furthermore, assume that $J_i/J_{i-1}$ is an idempotent ideal in $A/J_{i-1}$ (meaning $(J_i/J_{i-1})^2 = J_i/J_{i-1}$),  that it contains an idempotent in $A/J_{i-1}$ and that the radical of every $B_j$  is zero. Then,  Koenig and Xi show that $A$ has finite global dimension if and only if every $B_i$ has finite global dimension.

\section{Hecke algebras and Kazhdan-Lusztig cells}
In this section $(W,S)$ denotes an arbitrary Coxeter system (with $|S|<\infty$) together with a positive weight function $L$.  A positive weight function is a function $L:W\rightarrow\nN$ such that $L(ww')=L(w)+L(w')$ whenever $\ell(ww')=\ell(w)+\ell(w')$ where $\ell$ denotes the usual length function on $W$. The main reference is \cite{bible}.


\subsection{Hecke algebras and Kazhdan-Lusztig basis}
Let $\cA=\nC[v,v^{-1}]$ where $v$ is an indeterminate. Let $\cH$ be the Iwahori-Hecke algebra associated to $W$, with $\cA$-basis  $\{T_{w}|w\in W\}$ and multiplication rule given by
\begin{equation*}
T_{s}T_{w}=
\begin{cases}
T_{sw}, & \mbox{if } \ell(sw)>\ell(w),\\
T_{sw}+(v^{L(s)}-v^{-L(s)})T_{w}, &\mbox{if } \ell(sw)<\ell(w),
\end{cases}
\end{equation*}
for all $s\in S$ and $w\in W$. Let $\bar\ $ be the ring involution of $\cA$ which takes $v$ to $v^{-1}$. It can be extended to a ring involution of $\cH$ via
$$\ov{\sum_{w\in W} a_{w}T_{w}}=\sum_{w\in W} \bar{a}_{w}T_{w^{-1}}^{-1}\quad (a_{w}\in\cA).$$
We set 
$$\begin{array}{ccccccc}
\cA_{<0}&=v^{-1}\nZ[v^{-1}] &\text{ and }& \cH_{<0}=\bigoplus_{w\in W} \cA_{<0}T_{w}.
\end{array}$$
For each $w\in W$ there exists a unique element $C_{w}\in\cH$ (see \cite[Theorem 5.2]{bible}) such that
\begin{enumerate}
\item $\bar{C}_{w}=C_{w}$
\item $C_{w}\equiv T_{w} \mod \cH_{<0}$.
\end{enumerate}
For any $w\in W$ we set 
$$C_{w}=T_{w}+\sum_{y\in W } P_{y,w} T_{y}\quad \text{where $P_{y,w}\in \cA_{< 0}$}.$$
It is well known (\cite[\S 5.3]{bible}) that $P_{y,w}=0$ whenever $y\nleq w$  (here $\leq$ denotes the Bruhat order). It follows that $\{C_{w}|w\in W\}$ forms an $\cA$-basis of $\cH$ (the ``Kazhdan-Lusztig basis''). The coefficients $P_{y,w}$ are known as the Kazhdan-Lusztig polynomials. 

\begin{Def}\label{flat}
Following Lusztig \cite[\S 3.4]{bible}, there exists a unique involutive antiautomorphism, i.e. an $\cA$-involution, $\flat:\cH\longrightarrow\cH$ which carries $T_{w}$ to $T_{w^{-1}}$.  
\end{Def}

\begin{Rem}\label{flataction}
Using this map, we obtain right handed version of the multiplication of $\cH$:
\begin{equation*}
T_{w}T_{s}=
\begin{cases}
T_{ws}, & \mbox{if } \ell(ws)>\ell(w),\\
T_{ws}+(v^{L(s)}-v^{-L(s)})T_{w}, &\mbox{if } \ell(ws)<\ell(w).
\end{cases}
\end{equation*}
Further, since $\flat$ sends $\cH_{<0}$ to itself it can be shown that \cite[\S 5.6]{bible} that $C_{w}^{\flat}=C_{w^{-1}}$, from where it follows that  
$$P_{y,w}=P_{y^{-1},w^{-1}}.$$
\end{Rem}

\subsection{Kazhdan-Lusztig cells}
We denote by $h_{x,y,z}$ the structure constant with respect to the Kazhdan-Lusztig basis. That is,  we set

$$C_{x}C_{y}=\sum_{z\in W} h_{x,y,z}C_{z}.$$

Note that $\bar{h}_{x,y,z}=h_{x,y,z}$ and, similarly to Remark \ref{flataction}, we have  $h_{x,y,z}=h_{y^{-1},x^{-1},z^{-1}}$. We write $z\leftarrow_{\cL} y$ if there exists some $s\in S$ such that $h_{s,y,z}\neq 0$, that is $C_{z}$ appears with a non-zero coefficient in the expression of $C_{s}C_{y}$ in the Kazhdan-Lusztig basis. The Kazhdan-Lusztig left pre-order $\leq_{\cL}$ on $W$ is the transitive closure of this relation. 
The equivalence relation associated to $\leq_{\cL}$ will be denoted by $\sim_{\cL}$, that is

$$x\sim_{\cL}y \Longleftrightarrow x\leq_{\cL}y\text{ and } y\leq_{\cL} x\quad (x,y\in W).$$

The corresponding equivalence classes are called the left cells of $W$. Similarly, we can define a pre-order $\leq_{\cR}$ multiplying on the right in the defining relation. The associated equivalence relation will be denoted by $\sim_{\cR}$ and the corresponding equivalence classes are called the right cells of $W$. Using the antiautomorphism $\flat$, we have (see \cite[\S 8]{bible}) 

$$x\leq_{\cL} y \Longleftrightarrow x^{-1}\leq_{\cR} y^{-1}.$$

Finally we write $x\leq_{\cLR} y$ if there exists a sequence $x=x_{0},x_{1},...,x_{n}=y$ of $W$ such that for each $0\leq i\leq n-1$ we have either $x_{i}\leftarrow_{\cL} x_{i+1}$ or $x_{i}\leftarrow_{\cR} x_{i+1}$. The equivalence relation associated to $\leq_{\cLR}$ will be denoted by $\sim_{\cLR}$ and the equivalence classes are called the two-sided cells of $W$. \\
The preorders $\leq_{\cL},\leq_{\cR},\leq_{\cLR}$ induce partial orders on the left, right and two-sided cells, respectively.


\subsection{Kazhdan-Lusztig cell modules}
\label{cell-mod}
We will follow the notation of \cite{semi}. Let $?\in\{\cL,\cR,\cLR\}$. We define a $?$-ideal to be a left ideal if $?=\cL$, a right ideal if $?=R$ and a two-sided ideal if $?=\cLR$. Similarly, an $\cH$-?-module is a left $\cH$-module if $?=\cL$, a right $\cH$-module if $?=R$ and an $\cH$-$\cH$-bimodule if $?=\cLR$. 

Let $\Gamma$ be a $?$-cell of $W$.  We set

$$\begin{array}{rllcc}
\cH_{\leq_{?}\Gamma}&=\sg C_{y}\mid y\leq_{?}w, w\in\Gamma \sd_{\cA} \text{ and }\\
\cH_{<_{?}\Gamma}&=\sg C_{y}\mid y<_{?}w, w\in\Gamma \sd_{\cA}.\\
\end{array}$$

Then by definition of  $\leq_{?}$, we see that  $\cH_{\leq_{?}\Gamma}$ and $\cH_{<_{?}\Gamma}$ are $?$-ideals of $\cH$. Therefore 
$$\cM^{?}_{\Gamma}:=\cH_{\leq_{?}\Gamma}/\cH_{<{?}\Gamma}$$  is naturally an $\cH$-$?$-module. It is called the Kazhdan-Lusztig cell module associated to $\Gamma$. Note that it is a free $\cA$-module with basis the images of the elements $C_{w}$ for $w\in \Gamma$.

Let $\Gamma$ be a two-sided cell of $W$ and let 
$$\Gamma=\bigcup^{m}_{i=1} \Gamma_{i}$$
be its decomposition into left cells. We denote by $[\ .\ ]$ the natural projection onto $\cM^{\cLR}_{\Gamma}$. (Then $\cM^{\cLR}_{\Gamma}$ has an $\cA$-basis $\{[C_{w}]\mid w\in \Gamma\}$.) As an $\cA$-module we have

$$\cM^{\cLR}_{\Gamma}\cong\bigoplus_{i=1}^{m} \cM^{\cL}_{\Ga_{i}}.$$
Further if one assumes that the Lusztig conjecture
\begin{equation*}
x\leq_{\cL} y \text{ and } x\sim_{\cLR} y \Longrightarrow x\sim_{\cL}y\tag{$\ast$}
\end{equation*}

holds, then the $\cA$-submodules $\cM^{\cL}_{\Gamma_{i}}$ are left $\cH$-submodules of $\cM^{\cLR}_{\Gamma}$ for all $1\leq i\leq m$. Indeed, for all $h\in\cH$ and $y \in \Gamma_{i}$, we have  
$$hC_{y}=\sum_{z\leq_{\cL}y} a_{z} C_{z}=\sum_{z\leq_{\cL} y, z\in \Gamma}a_{z} C_{z}+\sum_{z\leq_{\cL}y,z\notin \Gamma}a_{z} C_{z}\quad\text{for some $a_{z}\in \cA$}$$ 
which, using $(\ast)$, yields that $h[C_{y}]\in\cM^{\cL}_{\Gamma_{i}}$.

\begin{Rem}
Conjecture $(\ast)$ is known to hold in the equal parameter case. In \cite{jeju4} it is shown that it holds in affine Weyl groups of rank 2 for all choices of parameters. 
\end{Rem}

Finally, since the $\cH$-$\cH$-bimodule $\cM^{\cLR}_{\Gamma}$ is a two-sided ideal in $\cH/ \cH_{<_{\cLR} \Gamma}$, it  can be viewed as an algebra (possibly without identity element). The multiplication is given by 
$$[C_{x}][C_{y}]=[C_{x}C_{y}]=\sum_{z\in\Gamma} h_{x,y,z}[C_{z}].$$
If one assumes that $(\ast)$ holds, then the $\cA$-submodules $\cM^{\cL}_{\Gamma_{i}}$ are left ideals of $\cM^{\cLR}_{\Gamma}$ for all $1\leq i\leq m$. 

\begin{Rem}
Note that the involution $\flat$ fixes each $\cH_{\leq_{\cLR}\Ga}$ and $\cH_{<_{\cLR}\Ga}$ hence induces an involution on $\cM^{\cLR}_{\Ga}$. We will still denote this involution by $\flat$.
\end{Rem}


\subsection{Generic parameters}
\label{generic}
Let $\bar{S}=\{\bar{{\bf s}}_{1},...,\bar{{\bf s}}_{r}\}$ be the set of conjugacy classes in $S$. Any weight function on $W$ is completely determined by its values on $\bar{S}$. Let $V=\nR^{r}$ be the Euclidean space of dimension $r$ and let $\omega_{1},...,\omega_{r}$ be the standard basis of $V$.  We identify the set of weight functions on $W$ with the set of points in $V$ with integer coordinates via
$$L\longrightarrow (L(s_{1}),...,L(s_{r}))\in V$$
where $s_{i}\in \bar{{\bf s}}_{i}$ for all $i$. The element of the $r$-tuple $(L(s_{1}),...,L(s_{r}))$ are called the parameters. To any choice of parameters one can associate a partition of $W$ into left, right and two-sided cells. 

According to Bonnaf\'e's semicontinuity conjecture  \cite{semi}, there exists a minimal finite set of hyperplanes $\fH$ in $V$ such that the partition of $W$ into cells is the same for all parameters $\cP,\cP'\in\nN^{r}$ belonging to the same $\fH$-facet (we refer to \cite{Bourbaki} for the definition of facets). The elements of this minimal set are called essential hyperplanes. The conjecture also states that the partition into cells for non-generic parameters can be recovered from the partition with respect to generic parameters. We refer to \cite{semi} for details on this conjecture.  In the following definition we assume that Bonnaf\'e's conjecture holds. 
\begin{Def}
The parameters $\cP:=(a_{1},\ldots,a_{r})\in\nN^{r}$ are called generic if they do not belong to any essential hyperplane for $W$.
\end{Def}
In this paper, we are only concerned with affine Weyl groups of rank 2 where the semicontinuity conjecture is known to hold and where the generic parameters have been determined in \cite{jeju4}.

\begin{Exa}
\label{G2}
Let $(W,S)$ be the affine Weyl group of type $\tG_{2}$ with diagram and weight function given by 
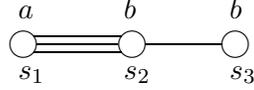
\begin{center}
\begin{picture}(150,32)
\put( 40, 10){\circle{10}}
\put( 44,  7){\line(1,0){33}}
\put( 45,  10){\line(1,0){30.5}}
\put( 44, 13){\line(1,0){33}}
\put( 81, 10){\circle{10}}
\put( 86, 10){\line(1,0){29}}
\put(120, 10){\circle{10}}
\put( 38, 20){$a$}
\put( 78, 20){$b$}
\put(118, 20){$b$}
\put( 38, -3){$s_{1}$}
\put( 78, -3){$s_{2}$}
\put(118,-3){$s_{3}$}
\end{picture}
\end{center}
\vspace{.05cm}
where $a,b$ are positive integers. Then $(a,b)\in\nN^{2}$ is generic if and only if $a/b\neq 1,\ 3/2,\ 2$. The corresponding partition into cells can be found in \cite{jeju4}.
\end{Exa}

\begin{Exa}
\label{B2}
Let $(W,S)$ be the affine Weyl group of type $\tB_{2}$ with diagram and weight function given by
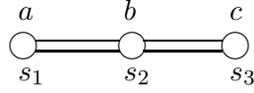
\begin{center}
\begin{picture}(150,32)
\put( 40, 10){\circle{10}}
\put( 44.5,  8){\line(1,0){31.75}}
\put( 44.5,  12){\line(1,0){31.75}}
\put( 81, 10){\circle{10}}
\put( 85.5,  8){\line(1,0){29.75}}
\put( 85.5,  12){\line(1,0){29.75}}
\put(120, 10){\circle{10}}
\put( 38, 20){$a$}
\put( 78, 20){$b$}
\put(118, 20){$c$}
\put( 38, -3){$s_{1}$}
\put( 78, -3){$s_{2}$}
\put(118,-3){$s_{3}$}
\end{picture}
\end{center}
\vspace{.05cm}
where $a,b,c$ are positive integers. We define the following hyperplanes in $\nR^{2}$
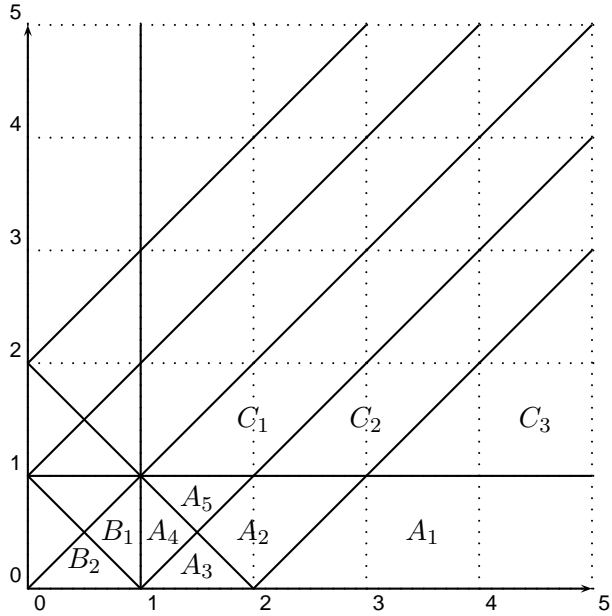
\begin{center}
\psset{xunit=1.5cm}
\psset{yunit=1.5cm}
\begin{pspicture}(0,-0.3)(5.5,5.3)
\psgrid[subgriddiv=1,griddots=10,gridlabels=8 pt](0,0)(5,5)

\psline(0,0)(5,5)
\psline(1,0)(1,1)
\psline(1,1)(5,1)
\psline(1,0)(0.5,.5)
\psline(1,0)(5,4)
\psline(2,0)(1,1)
\psline(2,0)(5,3)


\psline(0,1)(1,1)

\psline(.5,.5)(0,1)

\psline(1,1)(1,5)

\psline(1,1)(0,2)

\psline(0,2)(3,5)
\psline(0,1)(4,5)


\rput(3.5,.5){$A_{1}$}
\rput(2,.5){$A_{2}$}
\rput(1.5,.2){$A_{3}$}
\rput(1.2,.5){$A_{4}$}
\rput(1.5,.8){$A_{5}$}


\rput(4.5,1.5){$C_{3}$}
\rput(3,1.5){$C_{2}$}
\rput(2,1.5){$C_{1}$}

\rput(.5,.25){$B_{2}$}
\rput(.8,.5){$B_{1}$}

\psline{->}(0,0)(0,5)
\psline{->}(0,0)(5,0)
\end{pspicture}
\end{center}

Then  $(a,b,c)\in\nN^{3}$ is generic if and only if $(a/b,c/b)$ does not belong to any hyperplanes on the picture above \cite{jeju4,comp}. The corresponding partition into cells can be found in \cite{comp}. 
\end{Exa}


\subsection{On the induction of Kazhdan-Lusztig cells}
In this section we introduce the relative Kazhdan-Lusztig polynomials as in \cite{geck}. Let $S'\subsetneq S$. We denote by $W'$ the subgroup of $W$ generated by $S'$ and by $X'$ the set of distinguished left coset representative of $W'$ in $W$. Every element of $w\in W$ can be written uniquely $w=xu$ where $x\in X'$ and $u\in W'$. Note that $\ell(w)=\ell(x)+\ell(u)$. \\
Let $\preceq'$ be the Kazhdan-Lusztig left preorder relation defined with respect to the Coxeter group $(W',S')$ and the corresponding Hecke algebra. We define a relation $\sq$ on $W$ as follows. Let $x,y\in X'$ and $u,v\in W'$. We write $xu\sqs yv$ if $x<y$ and $u\preceq v$. We write $xu\sq yv$ if $xu\sqs yv$ or $x=y$ and $u=v$. Then we have
\begin{Prop}(\cite[Proposition 3.3]{geck})
\label{induction-st}
For any $y\in X'$ and $v\in W'$ we have
$$C_{yv}=\underset{xu\sq yv}{\sum_{x\in X',u\in W'}}p^{*}_{xu,yv}T_{x}C_{u}$$
where $p^{*}_{yv,yv}=1$ and $p_{xu,yv}\in\bA_{<0}$ if $xu\sqs yv$.
\end{Prop}

\begin{Rem}
Recall the $\cA$-involution $\flat:\cH\rightarrow\cH$ from Definition \ref{flat}, which can be used to obtain a right-handed version of the above result. First $Y'=X'^{-1}$ is the set of distinguished right coset representative of $W'$ in $W$. Any $w\in W$ can be uniquely written $w=ux$ where $u\in W'$ and $x\in Y'$. We get for all $y\in Y'$ and $v\in W'$:
$$C_{vy}=\underset{ux\sq_{\cR} vy}{\sum_{x\in Y',u\in W'}}p^{\ast,r}_{ux,vy}C_{u}T_{x}$$
where $p^{\ast,r}_{ux,vy}=p^{\ast}_{(ux)^{-1},(yv)^{-1}}$. 
\end{Rem}
Using this theorem, Geck obtained the following result (see \cite[Section 4]{geck}).
\begin{Cor}
\label{left-ideal}
Let $\fb$ be a left ideal of $W'$ with respect to $\leq'_{\cL}$. Then the set $X'\fb$ is a left ideal of $W$ with respect to $\leq_{\cL}$.
\end{Cor}

\subsection{Generalised induction of Kazhdan-Lusztig cells}
\label{gid}
We now introduce the Generalised Induction Theorem. The idea is to generalise the construction above to some subsets of $W$ which may not be parabolic subgroups. We refer to \cite{jeju3,jeju4} for details.\\

We consider a subset $U\subseteq W$ and a collection $\{X_{u}\ |\ u\in U\}$ of subsets of $W$ satisfying the following conditions
\begin{enumerate}
\item[{\bf I1}.] for all $u\in U$, we have $e\in X_{u}$,
\item[{\bf I2}.] for all $u\in U$ and $x\in X_{u}$ we have $\ell(xu)=\ell(x)+\ell(u)$,
\item[{\bf I3}.] for all $u,v\in U$ such that $u\neq v$ we have $X_{u}u\cap X_{v}v=\emptyset$,
\item[{\bf I4}.] the submodule $\cM:=\sg T_{x}C_{u}|\ u\in U,\ x\in X_{u}\sd_{\cA}\subseteq \cH$ is a left ideal.
\end{enumerate}
One can easily see that the set $\cB:=\{T_{x}C_{u}|u\in U,x\in X_{u}\}$ is an $\cA$-basis of $\cM$. Thus for all $y\in W$ and all $v\in U$, we can write
$$T_{y}C_{v}=\sum_{u\in U,x\in X_{u}}a_{x,u}T_{x}C_{u}\quad \text{for some $a_{x,u}\in \cA$}.$$
Let $\preceq$ be the relation on $U$ defined as follows. Let $u,v\in U$. We write $u\preceq v$ if there exist $y\in W$ and $x\in X_{u}$ such that $T_{x}C_{u}$ appears with a non-zero coefficient in the expression of $T_{y}C_{v}$ in the basis $\cB$. We still denote by $\preceq$ the pre-order induced by this relation (i.e. the transitive closure). For $u,v\in U$, $x\in X_{u}$ and $y\in X_{v}$ we write $xu\sqs yv$ if $u\preceq v$ and $xu< yv$. We write $xu\sq yv$ if $xu\sqs yv$ or $x=y$ and $u=v$.
\begin{Prop}
(\cite[Proposition 3.8]{jeju3})
\label{Gind}
For any $v\in U$ and $y\in X_{u}$, there exist a unique family of polynomials $(p^{*}_{xu,yv})_{xu\sqs yv}$ in $\bA_{<0}$ such that 
$$\tC_{yv}:=T_{y}C_{v}+\ind{xu\sqs yv}{\sum_{u\in U, x\in X_{u}}}p^{*}_{xu,yv}T_{x}C_{u}$$
is stable under the $\bar{\ }$ involution. 
\end{Prop}
\noindent
Now if one assumes that
\begin{enumerate}
\item[{\bf I5}.] for all $v\in U$, $y\in X_{v}$ we have
$$T_{y}C_{v}\equiv T_{yv} \mod \cH_{<0} $$
\end{enumerate}
then we have $\tC_{yv}=C_{yv}$.
\begin{Rem}
This is really a generalisation of Geck' s result. If we set $U=W'$ and $X_{u}=X'$ for all $u\in U$ then condition {\bf I1}--{\bf I4} are satisfied. For condition {\bf I5} we have for $v\in W'$ and $y\in X'$:
\begin{align*}
\tilde{C}_{yv}&=T_{y}C_{v}+\ind{xu\sqs yv}{\sum_{u\in U, x\in X_{u}}}p^{*}_{xu,yv}T_{x}C_{u}\\
&=T_{y}\big( T_{v}+\sum_{v_{1}<v}P_{v_{1},v}T_{v_{1}} \big)+
\ind{xu\sqs yv}{\sum_{u\in U, x\in X_{u}}}p^{*}_{xu,yv}T_{x}\sum_{u_{1}\leq u}P_{u_{1},u}T_{u_{1}}\\ 
&=T_{yv}+\big(\sum_{v_{1}<v}P_{v_{1},v}T_{y}T_{v_{1}} \big)+
\ind{xu\sqs yv}{\sum_{u\in U, x\in X_{u}}}\sum_{u_{1}\leq u}p^{*}_{xu,yv}P_{u_{1},u}T_{x}T_{u_{1}}\\
&=T_{yv}+\big(\sum_{v_{1}<v}P_{v_{1},v}T_{yv_{1}} \big)+
\ind{xu\sqs yv}{\sum_{u\in U, x\in X_{u}}}\sum_{u_{1}\leq u}p^{*}_{xu,yv}P_{u_{1},u}T_{xu_{1}}\\
&\equiv T_{yv} \mod \cH_{<0}.
\end{align*}
Since $\tC_{yv}$ is stable under the involution  $\bar{\ }$ we get $\tC_{yv}=C_{yv}$.
\end{Rem}


\section{Main result}
Throughout this section, $W$ denotes an affine Weyl group of rank 2 and $L$ a generic weight function. We fix a two-sided cell $\Gamma$ of $W$ and wish to show that $\cM^{\cLR}_{\Gamma}$ is an affine cell ideal in $\cH/\cH_{<_{\cLR}\Gamma}$. In Sections \ref{description}--\ref{main-result}, we assume that, in the case where $W$ is of type $\tG_{2}$, the two-sided cell $\Gamma$ is either infinite or does not intersect the group generated by $s_{2},s_{3}$. The remaining cases will be treated in Section \ref{finite}.\\

\subsection{Description of $\Gamma$} 
\label{description}
In this section we present a very nice description of $\Gamma$. It is rather surprising that most of the two-sided cells in $W$ can be described in such a uniform way. Note that this description is vital in the proof of affine cellularity. We refer to the next section for examples of this description.\\
 
By inspection of the different partitions into cells given in \cite{comp,jeju4}, one can show  that there exist two subsets $\cT$ and $\cZ=\{z_{1}=e,z_{2},\ldots,z_{m}\}$ of $W$ and an element $w_{\Gamma}\in W$ in a parabolic subgroup $W'$ such that for all  $1\leq i,j\leq m$ and all $\tau\in\cT$ we have 
$$\ell(z_{i}^{-1}\tau w_{\Gamma}z_{k})=\ell(z_{i}^{-1})+\ell(\tau)+\ell(w_{\Gamma})+\ell(z_{j})$$
 and the following map is bijective
$$\begin{array}{ccccccc}
&\cZ\times \cT\times\cZ &\longrightarrow & \Gamma&\\
& (z_{i},\tau,z_{j})&\longmapsto& z_{i}^{-1}\tau w_{\Gamma}z_{j}&\in\Gamma_{i}^{-1}\cap \Gamma_{j}.
\end{array}$$
The left cells lying in $\Gamma$ are of the form 
$$\Gamma_{j}=\{z_{i}^{-1}\tau w_{\Gamma}z_{j}\mid \tau \in \cT, 1\leq i\leq m\}.$$
and we have
$$(\Gamma_{i})^{-1}\cap \Gamma_{j}=\{z_{i}^{-1}\tau w_{\Gamma}z_{j}\mid \tau\in\cT\}.$$
The set $\cT$ can be expressed in one of the following forms:
\begin{enumerate} \label{cases}
\item $\cT=\{t_{1}^{n}t_{2}^{m}\mid n,m\in\nN\}$
\item $\cT=\{t^{n}\mid n\in\nN\}$
\item $\cT=\{e,t\}$
\item $\cT=\{e\}$ 
\end{enumerate}
In case (1) we set $T=\{t_{1},t_{2}\}$, in case (2) and (3) we set $T=\{t\}$ and in case (4) we set $T=\{e\}$. In all cases it can be checked that for $\tau \in \cT$, we have
\begin{equation*}
\tau w_{\Gamma}=w_{\Gamma}\tau^{-1}.
\end{equation*}
It is clear that if $\Gamma$ is infinite, then we must be in case (1) or (2) and if $\Gamma$ is finite then we must be in case (3) or (4).
We refer to Section 5 for examples of the different cases: when $\Gamma=\tc_{0}$ we are in case (1), when $\Ga=\tc_{1}$ or $\tc_{2}$ we are in case (2), when $\Gamma=\tc_{3}$ we are in case (3) and when $\Ga=\tc_{4}$ we are in case (4).  
\begin{Rem}
This description is no longer true for all cells if the parameters are non-generic which is why in this paper we are only considering generic parameters.
\end{Rem}

\noindent
\subsection{The element $\bP$}
\label{Ppol}

Let $X'$ be the set of distinguished left coset representative of $W'$ in $W$ and 
let $\leq^{'}_{\cL}$ be the Kazhdan-Lusztig left preorder relation defined with respect to $W'$. Then explicit computations show that either
\begin{enumerate}
\item $w_{\Gamma}$ is the longest element in $W'$; 
\item $w_{\Gamma}=sw'$ where $w'$ is the longest element of $W'$, $w_{\Gamma}^{2}=1$
and   $\{w_{\Gamma},w'\}$ is a left ideal of $W'$ with respect to $\leq'_{\cL}$.
\end{enumerate}

\medskip

\emph{The element $\bP$ in Case (1).} Assume we are in Case (1). By well-known properties of the longest element in a Coxeter group, for all $u\in W'$ we have $u\leq'_{\cL} w_{\Gamma}$ implies that $u=w_{\Gamma}$. Further, one can easily check from the definition of $\sqs$ that this implies that if $w\in W$ satisfies $w\sq yw_{\Gamma}$ then $w=xw_{\Gamma}$ with $x<y$ and $x\in X'$.
Thus by Proposition \ref{induction-st} we get for all $y\in X'$
\begin{align*}
C_{yw_{\Gamma}}&=T_{y}C_{w_{\Gamma}}+\sum_{x<y, x\in X'} p^{*}_{xw_{\Gamma},yw_{\Gamma}}T_{z}C_{w_{\Gamma}}\\
&=\big( T_{y}+\sum_{x<y, x\in X'} p^{*}_{xw_{\Gamma},yw_{\Gamma}}T_{z}\big)C_{w_{\Gamma}}.
\end{align*}
We set for $y\in\cZ^{-1}\cup T$
$$\bP(y)=\sum_{x\leq y, x\in X'}p^{\ast}_{xw_{\Gamma},yw_{\Gamma}}T_{x},\quad \bPr(y^{-1})=(\bP(y))^{\flat},$$
so that we have 
$$\bP(y)C_{w_{\Gamma}}=C_{yw_{\Gamma}}\text{ and } C_{w_{\Gamma}}\bPr(y^{-1})=C_{w_{\Gamma}y^{-1}}.$$
The last equality holds since
$$C_{w_{\Gamma}}\bPr(y^{-1})=C_{w_{\Gamma}}(\bP(y))^{\flat}=(\bP(y)C_{w_{\Gamma}})^{\flat}=(C_{yw_{\Gamma}})^{\flat}=C_{w_{\Gamma}^{-1}y^{-1}}=C_{w_{\Gamma}y^{-1}}.$$

\medskip

\emph{The element $\bP$ in Case (2).}  Assume that we are in Case (2). We set $U=\{w_{\Gamma}\}$ and $X=X'\cup X's$. 

\emph{Claim.} The set $U$ together with $X$ satisfy condition {\bf I1}--{\bf I5}.

 Assume for now that it is the case, then for all $y\in X$ we have (note that in the sum below the element $x$ is chosen in $X$ and not in $X'$ as above!)
\begin{align*}
C_{yw_{\Gamma}}&=T_{y}C_{w_{\Gamma}}+\sum_{x<y, x\in X} p^{*}_{xw_{\Gamma},yw_{\Gamma}}T_{x}C_{w_{\Gamma}}\\
&=\big( T_{y}+\sum_{x<y, x\in X_{c}} p^{*}_{xw_{\Gamma},yw_{\Gamma}}T_{z}\big)C_{w_{\Gamma}}.
\end{align*}
For $y\in\cZ^{-1}\cup T$ we set
$$\bP(y)=\sum_{x\leq y, x\in X}p^{\ast}_{xw_{\Gamma},yw_{\Gamma}}T_{x},\quad \bPr(y^{-1})=(\bP(y))^{\flat},$$
so that we have
$$\bP(y)C_{w_{\Gamma}}=C_{yw_{\Gamma}}\text{ and } C_{w_{\Gamma}}\bPr(y^{-1})=C_{w_{\Gamma}y^{-1}}.$$
The last equality holds since 
$$C_{w_{\Gamma}}\bPr(y^{-1})=C_{w_{\Gamma}}(\bP(y))^{\flat}=(\bP(y)C_{w_{\Gamma}})^{\flat}=(C_{yw_{\Gamma}})^{\flat}=C_{w_{\Gamma}^{-1}y^{-1}}=C_{w_{\Gamma}y^{-1}}$$ (here we need the fact that $w_{\Gamma}^{2}=1$).

\emph{Proof of Claim.} In order to verify that $U=\{w_{\Gamma}\}$ together with  $X=X'\cup X's$ satisfies condition {\bf I1}--{\bf I5}, note that $Xw_{\Gamma}=X'w'\cup X'w_{\Gamma}$, from where 
conditions {\bf I1}--{\bf I3} follow easily. Next we know that $\{w',w_{\Gamma}\}$ is a left ideal of $W'$, thus by Corollary \ref{left-ideal} we get that 
$$\fB=\sg C_{w}\mid w\in X'w_{\Gamma}\cup X'w'\sd_{\cA}=\sg C_{xw_{\Gamma}}\mid x\in X\sd$$ is a left ideal of $\cH$. Since 
$$T_{x}C_{w_{\Gamma}}=C_{xw_{\Gamma}}+\sum_{z<xw_{\Gamma},z\in \fB} C_{z},$$
it follows that $\fB=\sg T_{x}C_{w_{\Gamma}}\mid x\in X\sd$ and {\bf I4} follows. Let $x\in X$. If $x\in X'$ Condition {\bf I5} is clearly satisfied since $w_{\Gamma}\in W'$. So assume that $x\in X's$, that is $x=x's$ for some $x'\in X'$. Then we have
\begin{align*}
T_{x}C_{w_{\Gamma}}&=T_{x'}T_{s}C_{w_{\Gamma}}\\
&=T_{x'}(C_{w'}-v^{-L(s)}C_{w_{\Gamma}})\\
&=T_{x'}C_{w'}-v^{-L(s)}T_{x'}C_{w_{\Gamma}}.
\end{align*}  
The claim follows since $T_{x'}C_{w'}\equiv T_{x'w'} \mod \cH_{<0}$ and $T_{x'}C_{w_{\Gamma}}\equiv T_{x'w_{\Gamma}} \mod \cH_{<0}$.

\bigskip

Finally, in both Cases (1) and (2), we set
$$\begin{array}{rllc}
\bP(t_{1}^{m}t_{2}^{n})&=\bP(t_{1})^{m}\bP(t_{2})^{n}&\mbox{if $T=\{t_{1},t_{2}\}$}\\
\bP(t^{n})&=\bP(t)^{n}& \mbox{otherwise}.
\end{array}$$


\subsection{Properties of the element $\bP$}
\label{properties}
Recall that $\cM^{\cLR}_{\Gamma}$ can be viewed as an algebra and that the $\cA$-modules
$$\sg [C_{w}]\mid w\in \Ga_{j}\sd_{\cA}$$ 
are left ideals in $\cM^{\cLR}_{\Gamma}$ (see Section \ref{cell-mod}).  

\begin{Lem}
\label{left-right}
Let $\tau\in \cT$. We have
$$\bP(\tau)C_{w_{\Gamma}}=C_{w_{\Gamma}}\bP_{R}(\tau^{-1}).$$
\end{Lem}

\begin{proof}
Using the fact that for all $\tau\in\cT$ we have $\tau w_{\Ga}=w_{\Ga}\tau^{-1}$ we get
$$\bP(\tau)C_{w_{\Gamma}}=C_{\tau w_{\Gamma}}=C_{w_{\Gamma}\tau^{-1}}=C_{w_{\Gamma}}\bP_{R}(\tau^{-1}).$$
\end{proof}

We define the following $\cA$-submodule of  $\cM^{\cLR}_{\Gamma}$:
$$\cM_{\cT}=\sg [C_{\tau w_{\Gamma}}]\mid \tau\in\cT\sd_{\cA}.$$

\begin{Lem}
\label{mod-tau}
We have
$$\sg [C_{w}]\mid w\in\Gamma_{1}\sd_{\cA}\cap \sg [C_{w}]\mid w\in (\Gamma_{1})^{-1}\sd_{\cA}=\cM_{\cT}$$
\end{Lem}

\begin{proof}
This follows directly from the fact that $\sg [C_{w}]\mid w\in\Gamma_{1}\sd_{\cA}$ (respectively \\ $\sg [C_{w}]\mid w\in (\Gamma_{1})^{-1}\sd_{\cA}$) is a left ideal of $\cM^{\cLR}_{\Gamma}$ (respectively a right ideal of $\cM^{\cLR}_{\Gamma}$) and the equality
$$(\Gamma_{1})^{-1}\cap \Gamma_{1}=\{\tau w_{\Gamma}\mid \tau\in\cT\}.$$
\end{proof}

\begin{Lem}
\label{basis-tau}
The set $\{[\bP(\tau)C_{w_{\Gamma}}]\mid \tau\in\cT\}$ is an $\cA$-basis of $\cM_{\cT}$.
\end{Lem}

\begin{proof}
Since $\bP(\tau)C_{w_{\Gamma}}=C_{w_{\Gamma}}\bP_{R}(\tau^{-1})$ and 
$$\bP(\tau)C_{w_{\Gamma}}\in \cH C_{w_{\Gamma}}\subseteq \sum_{z\leq_{\cL} w_{\Ga}} \cA C_{z}\text{ and } C_{w_{\Gamma}}\bP_{R}(\tau^{-1})\in C_{w_{\Gamma}}\cH\subseteq \sum_{z\leq_{\cR} w_{\Ga}}\cA C_{z}$$
we have $[\bP(\tau)C_{w_{\Gamma}}]\in \cM_{\cT}$ by the previous lemma.  Then the result follows easily from the fact that
$$\bP(\tau)C_{w_{\Gamma}}=C_{\tau w_{\Gamma}}+\sum_{z<\tau w_{\Gamma}} \cA C_{z}.$$
\end{proof}

\begin{Lem}
\label{basis-m1}
The set  $\{[\bP(z_{j}^{-1})\bP(\tau)C_{w_{\Gamma}}]\mid 1\leq j\leq m, \tau\in\cT\}$ is an $\cA$-basis of $\cM^{\cL}_{\Gamma_{1}}$.
\end{Lem}

\begin{proof}
Since 
$$\bP(z_{j}^{-1})\bP(\tau)C_{w_{\Gamma}}\in \cH C_{w_{\Gamma}}\subseteq \sum_{z\leq_{\cL} w_{\Ga}}\cA C_{z}$$
we have $[\bP(z_{j}^{-1})\bP(\tau)C_{w_{\Gamma}}]\in\cM^{\cL}_{\Ga_{1}}$. Then the results follows easily from the fact that
$$\bP(z_{j}^{-1})\bP(\tau)C_{w_{\Gamma}}=C_{z_{i}^{-1}\tau w_{\Gamma}}+\sum_{z<z_{i}^{-1}\tau w_{\Gamma}}C_{z}.$$
\end{proof}

\subsection{Main result}
\label{main-result}

We are now ready to define the different ingredients needed in order to show that each cell ideal is affine cellular. Recall the definition of $T$ and $\cT$ in Section \ref{description}. As our principal ideal domain $k$, we choose $\cA$. We set 
$$B=
\begin{cases}
\cA[t_{1},t_{2}] &\mbox{if $T=\{t_{1},t_{2}\}$}\\
\cA[t]&\mbox{if $T=\{t\}\qquad$}\\
\cA[t]/(t^{2}-1)&\mbox{if $T=\{e,t\}$}\\
\cA&\mbox{if $T=\{e\}$.}
\end{cases}$$

Note that the monomials in $B$ corresponds to the elements of $\cT$ and we will use this identification freely. \\
For all $1\leq i,j\leq m$ we know, by Lemma \ref{mod-tau}, that
$$[C_{w_{\Gamma}z_{j}}][C_{z_{i}^{-1}w_{\Gamma}}]\in\cM_{\cT},$$ 
thus, by Lemma \ref{basis-tau}, we have
$$[C_{w_{\Gamma}z_{j}}][C_{z_{k}^{-1}w_{\Gamma}}]=\sum_{\tau \in\cT} a^{j,k}_{\tau}[\bP(\tau)C_{w_{\Gamma}}]\text{ where $a^{j,k}_{\tau}\in \cA$}.$$
Let $V$ be the free $\cA$-module of rank $m$ on basis $v_1, \dots , v_m$ and define the $\cA$-bilinear form $\varphi$ by
$$\begin{array}{ccccccc}
\varphi:&V\times V &\longrightarrow & B\\
& (v_{j},v_{k})&\longmapsto& \underset{\tau\in \cT}{\sum}a^{j,k}_{\tau}\tau.
\end{array}$$

This defines an algebra $\mathbb{A}(V,B,\varphi)\cong V \otimes_{\cA} B\otimes_{\cA} V$ with multiplication $\cA$-bilinearly extended from $(v_i \otimes \tau \otimes v_j)(v_k \otimes \tau' \otimes v_l) = v_i \otimes \tau \varphi(v_j,v_k)  \tau' \otimes v_l$  as in Section \ref{affcell}.

We now define a map $$\tilde\Phi: \mathbb{A}(V,B,\varphi) \ra \cH$$ by $$v_i \otimes \tau \otimes v_j \mapsto \bP(z_{i}^{-1})\bP(\tau)C_{w_{\Gamma}}\bP_{R}(z_{j})$$
for basis elements $v_i,v_j $ of V and $\tau \in \cT$.

We have
$$\bP(z_{i}^{-1})\bP(\tau)C_{w_{\Gamma}}\bP_{R}(z_{j})\in \cH C_{w_{\Gamma}}\cH\subseteq \sum_{z\leq_{\cLR} w_{\Gamma}} \cA C_{z}.$$
Hence, the image of $\tilde \Phi$ is contained in $\cH_{\leq_{\cLR} \Gamma}$ and we can compose $\tilde{\Phi}$ with the natural projection $[\ .\ ]$ onto $\cM^{\cLR}_{\Gamma}$. We obtain a map
$$\begin{array}{ccccccc}
\Phi:&\mathbb{A}(V,B,\varphi) &\longrightarrow & \cM^{\cLR}_{\Gamma}\\
& v_{i}\otimes \tau \otimes v_{j}&\longmapsto& [\bP(z_{i}^{-1})\bP(\tau)C_{w_{\Gamma}}\bP_{R}(z_{j})].
\end{array}$$
Note that when $B=\cA[t_{1},t_{2}]$ we have $t_{1}^{m}t_{2}^{n}=t_{2}^{n}t_{1}^{m}$ in $B$ thus for $\Phi$ to be well-defined we need to have
$$\Phi(v_{i}\otimes t_{1}^{m}t_{2}^{n}\otimes v_{j})=\Phi(v_{i}\otimes t_{2}^{n}t_{1}^{m}\otimes v_{j}),$$
that is 
$$[\bP(z_{i}^{-1})\bP(t_{1})^{m}\bP(t_{2})^{n}C_{w_{\Gamma}}\bP_{R}(z_{j})]=[\bP(z_{i}^{-1})\bP(t_{2})^{n}\bP(t_{1})^{m}C_{w_{\Gamma}}\bP_{R}(z_{j})].$$
To prove this, it is enough to show that 
$$[\bP(t_{1})\bP(t_{2})C_{w_{\Gamma}}]=[\bP(t_{2})\bP(t_{1})C_{w_{\Gamma}}].$$
This is checked by {\it explicit computation with} GAP.

\begin{Prop}\label{affine-cellular-ideal}

\begin{enumerate}
\item \label{isom}
The map $\Phi:\mathbb{A}(V,B,\varphi) \ra \cM^{\cLR}_{\Ga}$ is an isomorphism of $\cA$-algebras. 
\item \label{bimod} Using \eqref{isom} to define left and right $\cH$-module structures on  $\mathbb{A}(V,B,\varphi)$ by letting  $h \in \cH$ act, for $v,w \in V$ and $b \in B$, as $h(v\otimes b \otimes w) = \Phi^{-1}(h\Phi(v \otimes b\otimes w)$ and $(v\otimes b\otimes w)h = \Phi^{-1}(\Phi(v\otimes b\otimes w)h)$ respectively, $\Phi$ is an isomorphism of $\cH$-$\cH$-bimodules.
\item \label{invol} We have $\Phi(v\otimes b \otimes w)^{\flat}=\Phi(w \otimes b \otimes v)$ for $v,w \in V$ and $b \in B$.
\end{enumerate}
\end{Prop}

\begin{proof}
The map $\Phi$ is $\cA$-linear by definition. We have, for basis elements $v_i,v_j,v_k,v_l$ of $V$ and $t,t' \in \cT$,
\begin{align*}
&\Phi\big(v_{i}\otimes t\otimes v_{j}\big)\Phi\big(v_{k}\otimes t'\otimes v_{l}\big)\\
&=[\bP(z_{i}^{-1})\bP(t)C_{w_{\Gamma}}\bP_{R}(z_{j})][\bP(z_{k}^{-1})\bP(t')C_{w_{\Gamma}}\bP_{R}(z_{l})]\\
&=[\bP(z_{i}^{-1})\bP(t)C_{w_{\Gamma}z_{j}}\bP(z_{k}^{-1})C_{w_{\Gamma}}\bP_{R}(t'^{-1})\bP_{R}(z_{l})]\\
&=[\bP(z_{i}^{-1})\bP(t)C_{w_{\Gamma}z_{j}}C_{z_{k}^{-1}w_{\Gamma}}\bP_{R}(t'^{-1})\bP_{R}(z_{l})]\\
&=[\bP(z_{i}^{-1})\bP(t)\big(\sum_{\tau \in\cT} a^{j,k}_{\tau}\bP(\tau)\big)C_{w_{\Gamma}}\bP_{R}(t'^{-1})\bP_{R}(z_{l})]\\
&=[\bP(z_{i}^{-1})\bP(t)\big(\sum_{\tau \in\cT} a^{j,k}_{\tau}\bP(\tau)\big)\bP(t')C_{w_{\Gamma}}\bP_{R}(z_{l})]\\
&=\Phi\big(v_{i}\otimes t\varphi(v_{j},v_{k})t'\otimes v_{l}\big).
\end{align*}
So $\Phi$ is indeed a morphism of $\cA$-algebras. 
\begin{Rem}
In the case where $\cT=\{e,t\}$ we quotient out by $t^{2}-1$ in $B$ because we have
$$(\bP(t))^{2}C_{w_{\Gamma}}=C_{w_{\Gamma}}.$$
\end{Rem}

The fact that $\Phi$ is bijective follows easily from the fact that 
$$\bP(z_{i}^{-1})\bP(\tau)C_{w_{\Gamma}}\bP_{R}(z_{j})=C_{z_{i}^{-1}\tau w_{\Gamma}z_{j}}+\sum_{z<z_{i}^{-1}\tau w_{\Gamma}z_{j}}\cA C_{z}.$$ This completes the proof of \eqref{isom}.

Claim \eqref{bimod} follows directly from the definition and the fact that $\cM^{\cLR}_{\Gamma}$ is an $\cH$-$\cH$-bimodule.

To prove Claim \eqref{invol}
we let $v_i,v_j$ be basis elements of $V$ and $\tau \in \cT$, and check
\begin{align*}
\Phi(v_{i}\otimes\tau\otimes v_{j})^{\flat}&=[(\bP(z_{i}^{-1})\bP(\tau)C_{w_{\Gamma}}\bP_{R}(z_{j}))]^{\flat}\\
&=[(\bP(z_{i}^{-1})\bP(\tau)C_{w_{\Gamma}}\bP_{R}(z_{j}))^{\flat}]\\
&=[(\bP_{R}(z_{j}))^{\flat}(C_{w_{\Gamma}})^{\flat}(\bP(\tau))^{\flat}(\bP(z_{i}^{-1}))^{\flat}]\\
&=[\bP(z_{j}^{-1})C_{w_{\Gamma}}\bP_{R}(\tau^{-1})\bP_{R}(z_{i})]\\
&=[\bP(z_{j}^{-1})\bP(\tau)C_{w_{\Gamma}}\bP_{R}(z_{i})]\\
&=\Phi(v_{j}\otimes\tau\otimes v_{i}).
\end{align*}
The claim follows from $\cA$-linearity.
\end{proof}

\begin{Th}
\label{maintheorem}
$\cM^{\cLR}_{\Gamma}$ is an affine cell ideal in $\cH/\cH_{<_{\cLR} \Gamma}$ with the $\cA$-involution induced by $\flat$.
\end{Th}

\begin{proof}
According to Proposition \ref{descrip}, this follows from Proposition \ref{affine-cellular-ideal} by choosing the $\cA$-involution on $B$ to be the identity.
\end{proof}

\subsection{Remaining cases}
\label{finite}
Assume that $W$ is of type $\tG_{2}$ (as in Example \ref{G2}) and that $\Gamma$ be a finite two-sided cell which intersect the group generated by $s_{2},s_{3}$. Let
$$\Gamma=\bigcup_{i=1}^{m} \Gamma_i$$
be the decomposition of $\Gamma$ into left cells. \\

Assume that $a>b$. Then it can be checked by inspection that for all $i,j$ we have that $(\Gamma_{i})^{-1}\cap\Gamma_{j}$
only contains one element: we will denote it by  $w^{i,j}$.

Note that this implies that each left cell contains $m$ elements. Let $V$ be a $m$-dimensional $\cA$-module on basis $v_1,\dots, v_m$ and let $B=\cA$. Let

$$\begin{array}{ccccccc}
\varphi:&V\times V &\longrightarrow & B\\
& (v_{j},v_{k})&\longmapsto& a_{j,k}
\end{array}$$

where $a_{j,k}\in B$ is such that 

$$[C_{w^{1,j}}][C_{w^{k,1}}]=a_{j,k}[C_{w^{1,1}}].$$

Then it can be checked in each case that the map
$$\begin{array}{ccccccc}

\Phi:&\mathbb{A}(V,B,\varphi) &\longrightarrow & \cM^{\cLR}_{\Gamma}\\
& v_{i}\otimes 1 \otimes v_{j}&\longmapsto& [C_{w^{i,j}}]
\end{array}$$

satisfies the required properties. This is proved by explicit computation. \\

\medskip

Assume that $a<b$. Then $\Gamma=\Gamma_{1}\cup\Gamma_{2}\cup \Gamma_{3}$ where
\begin{align*}
\Gamma_{1}&=\{s_{3},s_{2}s_{3},s_{1}s_{2}s_{3},s_{2}s_{1}s_{2}s_{3},s_{3}s_{2}s_{1}s_{2}s_{3},s_{1}s_{2}s_{1}s_{2}s_{3}\},\\
\Gamma_{2}&=\{s_{2},s_{3}s_{2}, s_{1}s_{2},s_{2}s_{1}s_{2},s_{3}s_{2}s_{1}s_{2},s_{1}s_{2}s_{1}s_{2}\},\\
\Gamma_{3}&=\{s_{2}s_{1},s_{3}s_{2}s_{1}, s_{1}s_{2}s_{1},s_{2}s_{1}s_{2}s_{1},s_{3}s_{2}s_{1}s_{2}s_{1},s_{1}s_{2}s_{1}s_{2}s_{1}\}.
\end{align*}
Let $B=\cA[t]/(t^{2}-1)$ and let $V$ be a 3-dimensional $\cA$-module with basis $E:=\{v_{1},v_{2},v_{3}\}$. Let 

$$\phi: E\times \{1,t\}\times E\rightarrow \Gamma$$

be such that $\phi(v_{i},1,v_{j})$ (respectively $\phi(v_{i},t,v_{j})$) is the element of minimal (respectively of maximal) length in $(\Gamma_{i})^{-1}\cap \Gamma_{j}$. 
Then we define $\varphi: V\times V\rightarrow B$ by the following matrix 
$$
\left(\begin{array}{ccc}
a_{1}&a_{2}+t&1\\
a_{2}+t&a_{3}+a_{4}t&0\\
1&0&a_{1}
 \end{array}\right)
$$
where
$$\begin{array}{llllll}
a_{1}&=v^{b}+v^{-b}& a_{2}&=v^{a-b}+v^{b-a}\\
a_{3}&=v^{-b}+v^{2a-b}+v^{b-2a}+v^{b} & a_{4}&=v^{-a}+v^{a}.
\end{array}$$
Finally we set
$$\begin{array}{ccccccc}
\Phi:&\mathbb{A}(V,B,\varphi) &\longrightarrow & \cM^{\cLR}_{\Gamma}\\
& v_{i}\otimes 1 \otimes v_{j}&\longmapsto& [C_{\phi(v_{i},1,v_{j})}]\\
& v_{i}\otimes t \otimes v_{j}&\longmapsto& [C_{\phi(v_{i},t,v_{j})}].
\end{array}$$
Then it can be checked by explicit computations that $\cM^{\cLR}_{\Gamma}$ is affine cellular in the quotient $\cH/\cH_{<_{\cLR}\Gamma}$.

\subsection{Proof of Theorem \ref{mainthmintro}}

Theorem \ref{mainthmintro} now follows from the results in Subsections \ref{main-result} and \ref{finite}. Indeed, let $W$ be an affine Weyl group of rank 2 together with a generic weight function $L$ and let $\cH$ be the associated Hecke algebra. Consider the filtration of $\cH$ by two-sided ideals $\cH_{\leq_{\cLR}\Gamma}$ given by the partial order $\leq_{\cLR}$ on two-sided cells $\Ga$ of $W$. We need to show that the Kazhdan-Lusztig cell modules $\cM_{\Ga}^{\cLR}$ are isomorphic to some $\mathbb{A}(V,B,\varphi)$. First we show that ``most'' of the two-sided cells $\Ga$ of $W$ can be described as
$$\Ga=\{z^{-1}\tau w_{\Gamma}z'\mid \tau \in \cT, z,z'\in\cZ\}$$
for some subsets $\cT,\cZ$ of $W$ and $w_{\Ga}\in W$ (see  Sections \ref{description}, Section \ref{examples} and Appendix \ref{appendix}). Then, using the Generalised Induction Theorem of Kazhdan-Lusztig cells (see Section \ref{gid}), we define polynomials 
$\bP(\tau)\in\cA$ for all $\tau\in\cT$ (see Section \ref{Ppol}). Finally, we show that the map
$$\begin{array}{ccccccc}
\Phi:&\mathbb{A}(V,B,\varphi) &\longrightarrow & \cM^{\cLR}_{\Gamma}\\
& v_{i}\otimes \tau \otimes v_{j}&\longmapsto& [\bP(z_{i}^{-1})\bP(\tau)C_{w_{\Gamma}}\bP_{R}(z_{j})]
\end{array}$$
is an isomorphism of $\cA$-algebras (see Section \ref{main-result}). We then treat the case of the two-sided cells which cannot be described as above in Section \ref{finite}.   \\

Applying the results in Section \ref{affcell}, we obtain a parameterisation of simple modules of $\cH$: For each cell we obtain a simple module for every maximal ideal of the corresponding $\cA$-algebra $B$, defined at the start of Subsection \ref{main-result}. If we specialise to $\mathbb{C}(v) \otimes_{\cA}\cH$, the parametrisation is simply given by tuples $(a,b) \in \mathbb{C}(v)^2$ if $T=\{t_{1},t_{2}\}$, $a \in\mathbb{C}(v)$ if $T=\{t\}$, $\pm1$ if $T=\{e,t\}$, and $1$ if $T=\{e\}$.
Also it is clear that the affine $\cA$-algebras $B$ that appear in our construction satisfy $\mathrm{rad}(B)=0$ and have finite global dimension. Thus in order to prove that the affine Hecke algebra $\cH$ has finite global dimension, using the affine cellular structure, one would need to show that every $\cM^{\cLR}_{\Ga}$ is an idempotent ideal in $\cH/\cH_{<_{\cLR}\Ga}$ and that it contains an idempotent element in $\cH/\cH_{<_{\cLR}\Ga}$.

\begin{Rem}
The problem of finiteness of global dimension of affine Hecke algebras have already been addressed by Opdam and Solleveld in \cite{OS}. Using methods of harmonic analysis, they determined the (finite) global dimension of affine Hecke algebras in the case where $v$ is specialized to a positive real number. 
\end{Rem}


\section{Examples}
\label{examples}
The aim of this section is to provide some explicit examples of the sets $\cZ$ and $\cT$ as defined in Section \ref{description}.
Let $W$ be an affine Weyl group of type $\tG_{2}$ as in Example \ref{G2} together with some generic parameters $a,b$ such that $a/b>2$. 
The partition into cells in this case is shown in the following figure: the left cells are formed by the alcoves lying in the same connected component after removing the thick lines and the two-sided cells are the unions of all the left cells whose names share the same subscript. The alcove corresponding to the identity is denoted by $\tc_{6}^{1}$. We use the geometric presentation of $W$ as defined in \cite{Lus1}. 
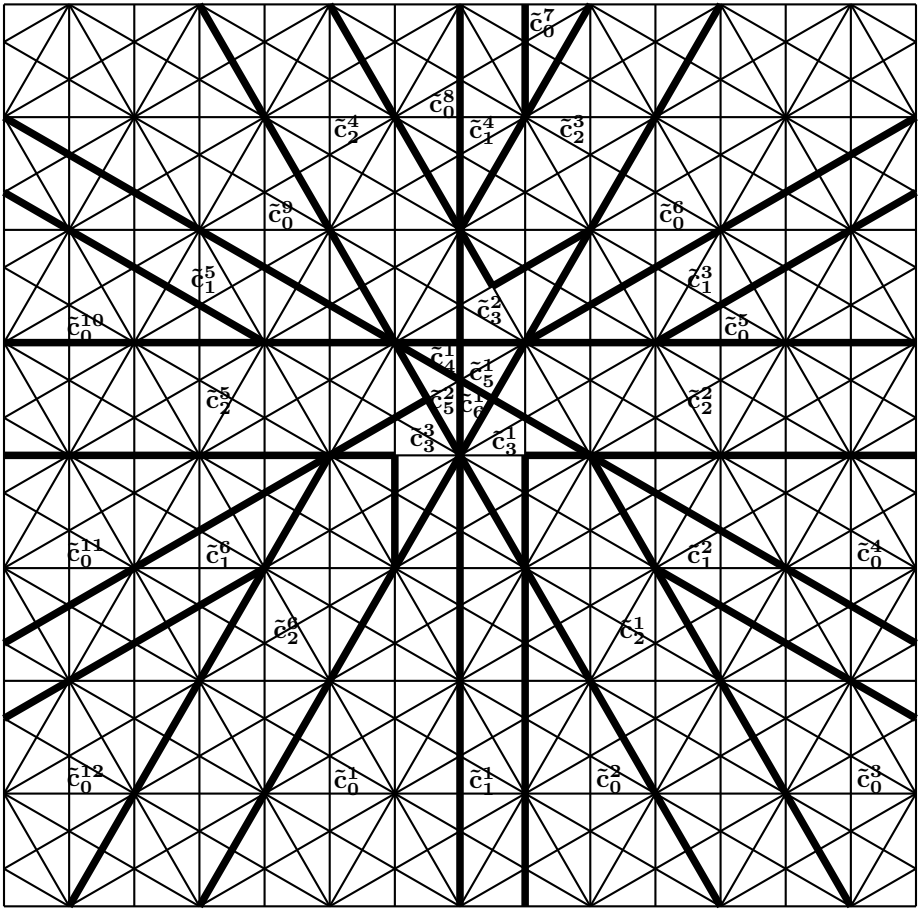
\begin{center}

\begin{pspicture}(-6,-6.4)(6,6.4)
\psset{unit=1cm}
\SpecialCoor
\psline(0,-6)(0,6)
\multido{\n=1+1}{7}{
\psline(\n;30)(\n;330)}
\rput(1;30){\psline(0,0)(0,5.5)}
\rput(1;330){\psline(0,0)(0,-5.5)}
\rput(2;30){\psline(0,0)(0,5)}
\rput(2;330){\psline(0,0)(0,-5)}
\rput(3;30){\psline(0,0)(0,4.5)}
\rput(3;330){\psline(0,0)(0,-4.5)}
\rput(4;30){\psline(0,0)(0,4)}
\rput(4;330){\psline(0,0)(0,-4)}
\rput(5;30){\psline(0,0)(0,3.5)}
\rput(5;330){\psline(0,0)(0,-3.5)}
\rput(6;30){\psline(0,0)(0,3)}
\rput(6;330){\psline(0,0)(0,-3)}
\rput(7;30){\psline(0,0)(0,2.5)}
\rput(7;330){\psline(0,0)(0,-2.5)}
\multido{\n=1+1}{7}{
\psline(\n;150)(\n;210)}
\rput(1;150){\psline(0,0)(0,5.5)}
\rput(1;210){\psline(0,0)(0,-5.5)}
\rput(2;150){\psline(0,0)(0,5)}
\rput(2;210){\psline(0,0)(0,-5)}
\rput(3;150){\psline(0,0)(0,4.5)}
\rput(3;210){\psline(0,0)(0,-4.5)}
\rput(4;150){\psline(0,0)(0,4)}
\rput(4;210){\psline(0,0)(0,-4)}
\rput(5;150){\psline(0,0)(0,3.5)}
\rput(5;210){\psline(0,0)(0,-3.5)}
\rput(6;150){\psline(0,0)(0,3)}
\rput(6;210){\psline(0,0)(0,-3)}
\rput(7;150){\psline(0,0)(0,2.5)}
\rput(7;210){\psline(0,0)(0,-2.5)}
\multido{\n=1.5+1.5}{4}{
\psline(-6.062,\n)(6.062,\n)}
\multido{\n=0+1.5}{5}{
\psline(-6.062,-\n)(6.062,-\n)}
\psline(0;0)(7;30)

\rput(0,1){\psline(0;0)(7;30)}
\rput(0,1){\psline(0;0)(7;210)}
\rput(0,2){\psline(0;0)(7;30)}
\rput(0,2){\psline(0;0)(7;210)}
\rput(0,3){\psline(0;0)(6;30)}
\rput(0,3){\psline(0;0)(7;210)}
\rput(0,4){\psline(0;0)(4;30)}
\rput(0,4){\psline(0;0)(7;210)}
\rput(0,5){\psline(0;0)(2;30)}
\rput(0,5){\psline(0;0)(7;210)}
\rput(0,6){\psline(0;0)(7;210)}
\rput(-1.732,6){\psline(0;0)(5;210)}
\rput(-3.464,6){\psline(0;0)(3;210)}
\rput(-5.196,6){\psline(0;0)(1;210)}

\rput(0,-1){\psline(0;0)(7;30)}
\rput(0,-1){\psline(0;0)(7;210)}
\rput(0,-2){\psline(0;0)(7;30)}
\rput(0,-2){\psline(0;0)(7;210)}
\rput(0,-3){\psline(0;0)(7;30)}
\rput(0,-3){\psline(0;0)(6;210)}
\rput(0,-4){\psline(0;0)(7;30)}
\rput(0,-4){\psline(0;0)(4;210)}
\rput(0,-5){\psline(0;0)(7;30)}
\rput(0,-5){\psline(0;0)(2;210)}
\rput(0,-6){\psline(0;0)(7;30)}
\rput(1.732,-6){\psline(0;0)(5;30)}
\rput(3.464,-6){\psline(0;0)(3;30)}
\rput(5.196,-6){\psline(0;0)(1;30)}

\psline(0;0)(6.928;60)
\rput(1.732,0){\psline(0;0)(6.928;60)}
\rput(3.464,0){\psline(0;0)(5.196;60)}
\rput(5.196,0){\psline(0;0)(1.732;60)}
\rput(1.732,0){\psline(0;0)(6.928;240)}
\rput(3.464,0){\psline(0;0)(6.928;240)}
\rput(5.196,0){\psline(0;0)(6.928;240)}
\rput(6.062,-1.5){\psline(0;0)(5.196;240)}
\rput(6.062,-4.5){\psline(0;0)(1.732;240)}
\rput(-1.732,0){\psline(0;0)(6.928;60)}
\rput(-3.464,0){\psline(0;0)(6.928;60)}
\rput(-5.196,0){\psline(0;0)(6.928;60)}
\rput(-1.732,0){\psline(0;0)(6.928;240)}
\rput(-3.464,0){\psline(0;0)(5.196;240)}
\rput(-5.196,0){\psline(0;0)(1.732;240)}
\rput(-6.062,1.5){\psline(0;0)(5.196;60)}
\rput(-6.062,4.5){\psline(0;0)(1.732;60)}

\psline(0;0)(6.928;120)
\psline(0;0)(6.928;300)
\rput(-1.732,0){\psline(0;0)(6.928;120)}
\rput(-3.464,0){\psline(0;0)(5.196;120)}
\rput(-5.196,0){\psline(0;0)(1.732;120)}
\rput(-1.732,0){\psline(0;0)(6.928;300)}
\rput(-3.464,0){\psline(0;0)(6.928;300)}
\rput(-5.196,0){\psline(0;0)(6.928;300)}
\rput(-6.062,-1.5){\psline(0;0)(5.196;300)}
\rput(-6.062,-4.5){\psline(0;0)(1.732;300)}
\rput(1.732,0){\psline(0;0)(6.928;300)}
\rput(3.464,0){\psline(0;0)(5.196;300)}
\rput(5.196,0){\psline(0;0)(1.732;300)}
\rput(1.732,0){\psline(0;0)(6.928;120)}
\rput(3.464,0){\psline(0;0)(6.928;120)}
\rput(5.196,0){\psline(0;0)(6.928;120)}
\rput(6.062,1.5){\psline(0;0)(5.196;120)}
\rput(6.062,4.5){\psline(0;0)(1.732;120)}

\rput(0,1){\psline(0;0)(7;150)}
\rput(0,1){\psline(0;0)(7;330)}
\rput(0,2){\psline(0;0)(7;150)}
\rput(0,2){\psline(0;0)(7;330)}
\rput(0,3){\psline(0;0)(6;150)}
\rput(0,3){\psline(0;0)(7;330)}
\rput(0,4){\psline(0;0)(4;150)}
\rput(0,4){\psline(0;0)(7;330)}
\rput(0,5){\psline(0;0)(7;330)}
\rput(0,6){\psline(0;0)(7;330)}
\rput(-1.732,6){\psline(0;0)(2;330)}

\rput(0,-1){\psline(0;0)(7;150)}
\rput(0,-1){\psline(0;0)(7;330)}
\rput(0,-2){\psline(0;0)(7;150)}
\rput(0,-2){\psline(0;0)(7;330)}
\rput(0,-3){\psline(0;0)(7;150)}
\rput(0,-3){\psline(0;0)(6;330)}
\rput(0,-4){\psline(0;0)(7;150)}
\rput(0,-4){\psline(0;0)(4;330)}
\rput(0,-5){\psline(0;0)(7;150)}
\rput(0,-5){\psline(0;0)(2;330)}
\rput(0,-6){\psline(0;0)(7;150)}

\rput(6.062,3.5){\psline(0;0)(5;150)}
\rput(6.062,4.5){\psline(0;0)(3;150)}
\rput(6.062,5.5){\psline(0;0)(1;150)}

\rput(-6.062,-3.5){\psline(0;0)(5;330)}
\rput(-6.062,-4.5){\psline(0;0)(3;330)}
\rput(-6.062,-5.5){\psline(0;0)(1;330)}

\psline(0;0)(7;330)
\psline(0;0)(7;150)
\psline(0;0)(7;210)
\psline(0;0)(6.928;240)


\psline[linewidth=1mm](0,1)(0,1.5)
\psline[linewidth=1mm](0,1)(.433,.75)
\psline[linewidth=1mm](0,0)(.866,1.5)
\psline[linewidth=1mm](.866,1.5)(0,1.5)
\psline[linewidth=1mm](-.866,1.5)(0,1)
\psline[linewidth=1mm](0,1)(0,0)
\psline[linewidth=1mm](0,0)(-.866,1.5)

\psline[linewidth=1mm](0,1.5)(0,3)
\psline[linewidth=1mm](0,3)(-1.732,6)
\psline[linewidth=1mm](-.866,1.5)(-3.464,6)

\psline[linewidth=1mm](.433,.75)(1.732,0)
\psline[linewidth=1mm](1.732,0)(6.062,0)
\psline[linewidth=1mm](2.598,1.5)(6.062,1.5)

\psline[linewidth=1mm](-.433,.75)(-1.732,0)
\psline[linewidth=1mm](-1.732,0)(-6.062,0)
\psline[linewidth=1mm](-2.598,1.5)(-6.062,1.5)

\psline[linewidth=1mm](.866,0)(.866,-1.5)
\psline[linewidth=1mm](.866,-1.5)(3.464,-6)
\psline[linewidth=1mm](.866,0)(1.732,0)
\psline[linewidth=1mm](2.598,-1.5)(5.196,-6)

\psline[linewidth=1mm](-.866,0)(-.866,-1.5)
\psline[linewidth=1mm](-.866,-1.5)(-3.464,-6)
\psline[linewidth=1mm](-.866,0)(-1.732,0)
\psline[linewidth=1mm](-2.598,-1.5)(-5.196,-6)

\psline[linewidth=1mm](.433,2.25)(0,3)
\psline[linewidth=1mm](0.866,4.5)(1.732,6)
\psline[linewidth=1mm](.433,2.25)(1.732,3)
\psline[linewidth=1mm](1.732,3)(3.464,6)

\psline[linewidth=1mm](0,0)(0,-6)
\psline[linewidth=1mm](0,0)(.866,-1.5)
\psline[linewidth=1mm](.866,-1.5)(.866,-6)

\psline[linewidth=1mm](1.732,0)(6.062,-2.5)
\psline[linewidth=1mm](2.598,-1.5)(6.062,-3.5)

\psline[linewidth=1mm](-1.732,0)(-6.062,-2.5)
\psline[linewidth=1mm](-2.598,-1.5)(-6.062,-3.5)

\psline[linewidth=1mm](.866,1.5)(6.062,4.5)
\psline[linewidth=1mm](2.598,1.5)(6.062,3.5)

\psline[linewidth=1mm](-.866,1.5)(-6.062,4.5)
\psline[linewidth=1mm](-2.598,1.5)(-6.062,3.5)

\psline[linewidth=1mm](0,3)(0,6)
\psline[linewidth=1mm](.866,4.5)(.866,6)

\psline[linewidth=1mm](.866,1.5)(1.732,3)
\psline[linewidth=1mm](0,0)(-.866,-1.5)

\psline[linewidth=1mm](0,1.5)(-.866,1.5)


\psline[linewidth=1mm](1.732,0)(2.598,-1.5)
\psline[linewidth=1mm](.866,1.5)(2.598,1.5)
\psline[linewidth=1mm](-.866,1.5)(-2.598,1.5)
\psline[linewidth=1mm](-1.732,0)(-2.598,-1.5)
\psline[linewidth=1mm](0,3)(.866,4.5)

\rput(-1.5,-4.33){{\small ${\bf \tc_{0}^{1}}$}}
\rput(1.988,-4.31){{\small ${\bf \tc_{0}^{2}}$}}
\rput(5.452,-4.31){{\small ${\bf \tc_{0}^{3}}$}}
\rput(5.452,-1.31){{\small ${\bf \tc_{0}^{4}}$}}
\rput(3.69,1.7){{\small ${\bf \tc_{0}^{5}}$}}
\rput(2.824,3.2){{\small ${\bf \tc_{0}^{6}}$}}
\rput(1.1,5.75){{\small ${\bf \tc_{0}^{7}}$}}
\rput(-0.23,4.67){{\small ${\bf \tc_{0}^{8}}$}}
\rput(-2.372,3.2){{\small ${\bf \tc_{0}^{9}}$}}
\rput(-4.97,1.7){{\small ${\bf \tc_{0}^{10}}$}}

\rput(-4.97,-1.3){{\small ${\bf \tc_{0}^{11}}$}}
\rput(-4.97,-4.3){{\small ${\bf \tc_{0}^{12}}$}}

\rput(2.3,-2.33){{\small ${\bf  \tc_{2}^{1}}$}}
\rput(3.2,.73){{\small ${\bf \tc_{2}^{2}}$}}
\rput(1.5,4.34){{\small ${\bf \tc_{2}^{3}}$}}
\rput(-1.5,4.34){{\small ${\bf \tc_{2}^{4}}$}}
\rput(-3.2,.73){{\small ${\bf \tc_{2}^{5}}$}}
\rput(-2.3,-2.33){{\small ${\bf  \tc_{2}^{6}}$}}

\rput(0.3,-4.33){{\small ${\bf  \tc_{1}^{1}}$}}
\rput(3.2,-1.33){{\small ${\bf \tc_{1}^{2}}$}}
\rput(3.2,2.34){{\small ${\bf \tc_{1}^{3}}$}}
\rput(0.3,4.33){{\small ${\bf \tc_{1}^{4}}$}}
\rput(-3.4,2.34){{\small ${\bf \tc_{1}^{5}}$}}
\rput(-3.2,-1.33){{\small ${\bf \tc_{1}^{6}}$}}


\rput(-.22,1.3){{\small ${\bf \tc_{4}^{1}}$}}

\rput(.6,.19){{\small ${\bf \tc_{3}^{1}}$}}
\rput(.4,1.93){{\small ${\bf \tc_{3}^{2}}$}}
\rput(-.49,.23){{\small ${\bf \tc_{3}^{3}}$}}

\rput(.3,1.1){{\small ${\bf \tc_{5}^{1}}$}}
\rput(-.23,.73){{\small ${\bf \tc_{5}^{2}}$}}

\rput(.17,.68){{\small ${\bf \tc_{6}^{1}}$}}

\end{pspicture}
\end{center}
We first have a look at the lowest two-sided cell $\tc_{0}$. In Figure \ref{c0},  we show the elements of the set $\cZ=\{z_{1},\ldots,z_{12}\}$: these are the elements which correspond to the alcoves in dark gray. The alcoves lying in the box in light gray correspond to the elements in $\cZ^{-1}$. We set  $t_{1}=s_{2}s_{1}s_{2}s_{1}s_{2}s_{3}$, $t_{2}=s_{1}s_{2}s_{1}s_{2}s_{3}s_{1}s_{2}s_{1}s_{2}s_{3}$ and
$$\cT:=\{t_{1}^{n}t_{2}^{m}|n,m\in \nN\}.$$
Then we have
\begin{align*}
\tc_{0}&=\{z_{i}^{-1}\tau w_{0}z_{j}\mid \tau \in \cT, 1\leq i,j\leq 12\}\\
\tc_{0}^{j}&=\{z_{i}^{-1}\tau w_{0}z_{j}\mid \tau \in \cT, 1\leq j\leq 12\}.
\end{align*}
where $w_{0}=w_{1,2}$. This should be understood in the following way. Let $w=z_{i}^{-1}\tau w_{0}z_{j}\in\tc_{0}$. The element $z_{j}$ indicate in which connected component of $\tc_{0}$ the element $w$ lies. Then $\tau$ indicates in which translate of the box $w$ lies and finally $z_{i}^{-1}$ indicates where in the translate of box $w$ lies. This is explained in Figure \ref{c0}.
\begin{figure}[h!]
\caption{Description of $\tc_{0}$}
\label{c0}
\begin{textblock}{10.5}(3,8.1)
\begin{center}
\psset{unit=.6cm}
\begin{pspicture}(-6,-6.4)(6,6.4)
\psset{linewidth=.1mm}

\pspolygon[fillcolor=lightgray,fillstyle=solid](0,0)(0,3)(.866,4.5)(.866,1.5)

\psline[linewidth=.5mm](0,0)(0,-6)
\psline[linewidth=.5mm](0,0)(-3.46,-6)
\pspolygon[fillcolor=gray,fillstyle=solid](0,0)(0,1)(0.433,0.75)

\psline[linewidth=.5mm](.866,-1.5)(.866,-6)
\psline[linewidth=.5mm](.866,-1.5)(3.46,-6)
\pspolygon[fillcolor=gray,fillstyle=solid](.866,-1.5)(.866,-.5)(.433,-.75)

\psline[linewidth=.5mm](-2.598,-1.5)(-6.062,-3.5)
\psline[linewidth=.5mm](-2.598,-1.5)(-5.196,-6)
\pspolygon[fillcolor=gray,fillstyle=solid](-2.598,-1.5)(-2.165,-.75)(-1.732,-1)

\psline[linewidth=.5mm](-1.732,0)(-6.062,0)
\psline[linewidth=.5mm](-1.732,0)(-6.062,-2.5)
\pspolygon[fillcolor=gray,fillstyle=solid](-1.732,0)(-.866,0)(-.866,.5)

\psline[linewidth=.5mm](-.866,1.5)(-3.464,6)
\psline[linewidth=.5mm](-.866,1.5)(-6.062,4.5)
\pspolygon[fillcolor=gray,fillstyle=solid](-.866,1.5)(0,1)(-.433,.75)

\psline[linewidth=.5mm](-2.598,1.5)(-6.062,1.5)
\psline[linewidth=.5mm](-2.598,1.5)(-6.062,3.5)
\pspolygon[fillcolor=gray,fillstyle=solid](-2.598,1.5)(-1.732,1.5)(-1.732,1)

\psline[linewidth=.5mm](0,3)(0,6)
\psline[linewidth=.5mm](0,3)(-1.732,6)
\pspolygon[fillcolor=gray,fillstyle=solid](0,3)(0,2)(.433,2.25)

\psline[linewidth=.5mm](.866,4.5)(.866,6)
\psline[linewidth=.5mm](.866,4.5)(1.732,6)
\pspolygon[fillcolor=gray,fillstyle=solid](.866,4.5)(.866,3.5)(1.3,3.75)

\psline[linewidth=.5mm](.866,1.5)(3.464,6)
\psline[linewidth=.5mm](.866,1.5)(6.062,4.5)
\pspolygon[fillcolor=gray,fillstyle=solid](.866,1.5)(0,1)(.433,.75)

\psline[linewidth=.5mm](2.598,1.5)(6.062,1.5)
\psline[linewidth=.5mm](2.598,1.5)(6.062,3.5)
\pspolygon[fillcolor=gray,fillstyle=solid](2.598,1.5)(1.732,1.5)(1.732,1)

\psline[linewidth=.5mm](2.598,-1.5)(6.062,-3.5)
\psline[linewidth=.5mm](2.598,-1.5)(5.196,-6)
\pspolygon[fillcolor=gray,fillstyle=solid](1.732,0)(.866,0)(.866,.5)

\psline[linewidth=.5mm](1.732,0)(6.062,0)
\psline[linewidth=.5mm](1.732,0)(6.062,-2.5)
\pspolygon[fillcolor=gray,fillstyle=solid](2.598,-1.5)(2.165,-.75)(1.732,-1)

\SpecialCoor
\psline(0,-6)(0,6)
\multido{\n=1+1}{7}{
\psline(\n;30)(\n;330)}
\rput(1;30){\psline(0,0)(0,5.5)}
\rput(1;330){\psline(0,0)(0,-5.5)}
\rput(2;30){\psline(0,0)(0,5)}
\rput(2;330){\psline(0,0)(0,-5)}
\rput(3;30){\psline(0,0)(0,4.5)}
\rput(3;330){\psline(0,0)(0,-4.5)}
\rput(4;30){\psline(0,0)(0,4)}
\rput(4;330){\psline(0,0)(0,-4)}
\rput(5;30){\psline(0,0)(0,3.5)}
\rput(5;330){\psline(0,0)(0,-3.5)}
\rput(6;30){\psline(0,0)(0,3)}
\rput(6;330){\psline(0,0)(0,-3)}
\rput(7;30){\psline(0,0)(0,2.5)}
\rput(7;330){\psline(0,0)(0,-2.5)}
\multido{\n=1+1}{7}{
\psline(\n;150)(\n;210)}
\rput(1;150){\psline(0,0)(0,5.5)}
\rput(1;210){\psline(0,0)(0,-5.5)}
\rput(2;150){\psline(0,0)(0,5)}
\rput(2;210){\psline(0,0)(0,-5)}
\rput(3;150){\psline(0,0)(0,4.5)}
\rput(3;210){\psline(0,0)(0,-4.5)}
\rput(4;150){\psline(0,0)(0,4)}
\rput(4;210){\psline(0,0)(0,-4)}
\rput(5;150){\psline(0,0)(0,3.5)}
\rput(5;210){\psline(0,0)(0,-3.5)}
\rput(6;150){\psline(0,0)(0,3)}
\rput(6;210){\psline(0,0)(0,-3)}
\rput(7;150){\psline(0,0)(0,2.5)}
\rput(7;210){\psline(0,0)(0,-2.5)}
\multido{\n=1.5+1.5}{4}{
\psline(-6.062,\n)(6.062,\n)}
\multido{\n=0+1.5}{5}{
\psline(-6.062,-\n)(6.062,-\n)}
\psline(0;0)(7;30)

\rput(0,1){\psline(0;0)(7;30)}
\rput(0,1){\psline(0;0)(7;210)}
\rput(0,2){\psline(0;0)(7;30)}
\rput(0,2){\psline(0;0)(7;210)}
\rput(0,3){\psline(0;0)(6;30)}
\rput(0,3){\psline(0;0)(7;210)}
\rput(0,4){\psline(0;0)(4;30)}
\rput(0,4){\psline(0;0)(7;210)}
\rput(0,5){\psline(0;0)(2;30)}
\rput(0,5){\psline(0;0)(7;210)}
\rput(0,6){\psline(0;0)(7;210)}
\rput(-1.732,6){\psline(0;0)(5;210)}
\rput(-3.464,6){\psline(0;0)(3;210)}
\rput(-5.196,6){\psline(0;0)(1;210)}

\rput(0,-1){\psline(0;0)(7;30)}
\rput(0,-1){\psline(0;0)(7;210)}
\rput(0,-2){\psline(0;0)(7;30)}
\rput(0,-2){\psline(0;0)(7;210)}
\rput(0,-3){\psline(0;0)(7;30)}
\rput(0,-3){\psline(0;0)(6;210)}
\rput(0,-4){\psline(0;0)(7;30)}
\rput(0,-4){\psline(0;0)(4;210)}
\rput(0,-5){\psline(0;0)(7;30)}
\rput(0,-5){\psline(0;0)(2;210)}
\rput(0,-6){\psline(0;0)(7;30)}
\rput(1.732,-6){\psline(0;0)(5;30)}
\rput(3.464,-6){\psline(0;0)(3;30)}
\rput(5.196,-6){\psline(0;0)(1;30)}

\psline(0;0)(6.928;60)
\rput(1.732,0){\psline(0;0)(6.928;60)}
\rput(3.464,0){\psline(0;0)(5.196;60)}
\rput(5.196,0){\psline(0;0)(1.732;60)}
\rput(1.732,0){\psline(0;0)(6.928;240)}
\rput(3.464,0){\psline(0;0)(6.928;240)}
\rput(5.196,0){\psline(0;0)(6.928;240)}
\rput(6.062,-1.5){\psline(0;0)(5.196;240)}
\rput(6.062,-4.5){\psline(0;0)(1.732;240)}
\rput(-1.732,0){\psline(0;0)(6.928;60)}
\rput(-3.464,0){\psline(0;0)(6.928;60)}
\rput(-5.196,0){\psline(0;0)(6.928;60)}
\rput(-1.732,0){\psline(0;0)(6.928;240)}
\rput(-3.464,0){\psline(0;0)(5.196;240)}
\rput(-5.196,0){\psline(0;0)(1.732;240)}
\rput(-6.062,1.5){\psline(0;0)(5.196;60)}
\rput(-6.062,4.5){\psline(0;0)(1.732;60)}

\psline(0;0)(6.928;120)
\psline(0;0)(6.928;300)
\rput(-1.732,0){\psline(0;0)(6.928;120)}
\rput(-3.464,0){\psline(0;0)(5.196;120)}
\rput(-5.196,0){\psline(0;0)(1.732;120)}
\rput(-1.732,0){\psline(0;0)(6.928;300)}
\rput(-3.464,0){\psline(0;0)(6.928;300)}
\rput(-5.196,0){\psline(0;0)(6.928;300)}
\rput(-6.062,-1.5){\psline(0;0)(5.196;300)}
\rput(-6.062,-4.5){\psline(0;0)(1.732;300)}
\rput(1.732,0){\psline(0;0)(6.928;300)}
\rput(3.464,0){\psline(0;0)(5.196;300)}
\rput(5.196,0){\psline(0;0)(1.732;300)}
\rput(1.732,0){\psline(0;0)(6.928;120)}
\rput(3.464,0){\psline(0;0)(6.928;120)}
\rput(5.196,0){\psline(0;0)(6.928;120)}
\rput(6.062,1.5){\psline(0;0)(5.196;120)}
\rput(6.062,4.5){\psline(0;0)(1.732;120)}

\rput(0,1){\psline(0;0)(7;150)}
\rput(0,1){\psline(0;0)(7;330)}
\rput(0,2){\psline(0;0)(7;150)}
\rput(0,2){\psline(0;0)(7;330)}
\rput(0,3){\psline(0;0)(6;150)}
\rput(0,3){\psline(0;0)(7;330)}
\rput(0,4){\psline(0;0)(4;150)}
\rput(0,4){\psline(0;0)(7;330)}
\rput(0,5){\psline(0;0)(7;330)}
\rput(0,6){\psline(0;0)(7;330)}
\rput(-1.732,6){\psline(0;0)(2;330)}

\rput(0,-1){\psline(0;0)(7;150)}
\rput(0,-1){\psline(0;0)(7;330)}
\rput(0,-2){\psline(0;0)(7;150)}
\rput(0,-2){\psline(0;0)(7;330)}
\rput(0,-3){\psline(0;0)(7;150)}
\rput(0,-3){\psline(0;0)(6;330)}
\rput(0,-4){\psline(0;0)(7;150)}
\rput(0,-4){\psline(0;0)(4;330)}
\rput(0,-5){\psline(0;0)(7;150)}
\rput(0,-5){\psline(0;0)(2;330)}
\rput(0,-6){\psline(0;0)(7;150)}

\rput(6.062,3.5){\psline(0;0)(5;150)}
\rput(6.062,4.5){\psline(0;0)(3;150)}
\rput(6.062,5.5){\psline(0;0)(1;150)}

\rput(-6.062,-3.5){\psline(0;0)(5;330)}
\rput(-6.062,-4.5){\psline(0;0)(3;330)}
\rput(-6.062,-5.5){\psline(0;0)(1;330)}

\psline(0;0)(7;330)
\psline(0;0)(7;150)
\psline(0;0)(7;210)
\psline(0;0)(6.928;240)

\end{pspicture}
\end{center}
\end{textblock}

\begin{textblock}{10.5}(10,9.8)
\begin{center}
\psset{unit=.5cm}
\begin{pspicture}(-6,-1)(9,9.4)
\psset{linewidth=.1mm}

\pspolygon[fillcolor=lightgray,fillstyle=solid](0,0)(0,3)(.866,4.5)(.866,1.5)

\psline[linewidth=.5mm]{->}(0,0)(0,3)
\pspolygon[fillcolor=gray,fillstyle=solid](0,3)(0,6)(.866,7.5)(.866,4.5)

\psline[linewidth=.5mm]{->}(0,0)(.866,1.5)
\pspolygon[fillcolor=gray,fillstyle=solid](.866,1.5)(.866,4.5)(1.732,6)(1.732,3)

\psline(0,0)(5.196,9)
\psline(0,0)(0,9)

\psline(0,9)(5.196,9)
\psline(0,7.5)(4.33,7.5)
\psline(0,6)(3.464,6)
\psline(0,4.5)(2.598,4.5)
\psline(0,3)(1.732,3)
\psline(0,1.5)(.866,1.5)

\psline(0,1)(.433,.75)
\psline(0,2)(.866,1.5)
\psline(0,3)(1.3,2.25)
\psline(0,4)(1.732,3)
\psline(0,5)(2.196,3.75)
\psline(0,6)(2.598,4.5)
\psline(0,7)(3.031,5.25)
\psline(0,8)(3.464,6)
\psline(0,9)(3.9,6.75)
\psline(1.732,9)(4.33,7.5)
\psline(3.464,9)(4.8,8.25)

\psline(0,1)(.866,1.5)
\psline(0,2)(1.732,3)
\psline(0,3)(2.598,4.5)
\psline(0,4)(3.464,6)
\psline(0,5)(4.3,7.5)
\psline(0,6)(5.196,9)
\psline(0,7)(3.464,9)
\psline(0,8)(1.732,9)

\psline(0,3)(.866,1.5)
\psline(0,6)(1.732,3)
\psline(0,9)(2.598,4.5)
\psline(1.732,9)(3.464,6)
\psline(3.464,9)(4.3,7.5)

\psline(.866,1.5)(.866,9)
\psline(1.732,3)(1.732,9)
\psline(2.598,4.5)(2.598,9)
\psline(3.464,6)(3.464,9)
\psline(4.3,7.5)(4.3,9)

\psline(0,3)(3.464,9)
\psline(0,6)(1.732,9)
\rput(-.7,1.4){$t_{1}$}
\rput(1.4,1.4){$t_{2}$}

\end{pspicture}\end{center}
\end{textblock}
\end{figure}
$\ $\\

\vspace{7.3cm}
\noindent
We now have a look at the two-sided cell $\tc_{1}$. We set $w_{1}=s_{1}s_{2}s_{1}s_{2}s_{1}$. In Figure \ref{c1}, we show the elements of the set $\cZ=\{z_{1},\ldots,z_{6}\}$: these are the elements which correspond to the alcoves in dark gray. We set  $t=s_{1}s_{2}s_{1}s_{2}s_{3}$ and 
$$\cT:=\{t^{n}|n\in \nN\}.$$
Then we have
\begin{align*}
\tc_{1}&=\{z_{i}^{-1}\tau w_{1}z_{j}\mid \tau \in \cT, 1\leq i,j\leq 6\}\\
\tc_{1}^{j}&=\{z_{i}^{-1}\tau w_{1}z_{j}\mid \tau \in \cT, 1\leq j\leq 6\}.
\end{align*}
This is explained in Figure \ref{c1}.

\newpage

\begin{figure}[h!]
\caption{Description of the two-sided cell $\tc_{1}$}
\label{c1}
\begin{textblock}{10.5}(3,3.8)
\begin{center}
\psset{unit=.5cm}
\begin{pspicture}(-6,-7.4)(6,6.4)
\psset{linewidth=.1mm}

\pspolygon[fillstyle=solid,fillcolor=lightgray](0,0)(0,3)(.866,1.5)

\psline[linewidth=.5mm](0,0)(0,-6)
\psline[linewidth=.5mm](.866,-1.5)(.866,-6)
\psline[linewidth=.5mm](0,0)(.866,-1.5)
\pspolygon[fillcolor=gray,fillstyle=solid](0,0)(0,1)(0.433,0.75)

\psline[linewidth=.5mm](-2.598,-1.5)(-6.062,-3.5)
\psline[linewidth=.5mm](-1.732,0)(-6.062,-2.5)
\psline[linewidth=.5mm](-1.732,0)(-2.598,-1.5)
\pspolygon[fillcolor=gray,fillstyle=solid](-1.732,0)(-.866,0)(-.866,.5)

\psline[linewidth=.5mm](-2.598,1.5)(-6.062,3.5)
\psline[linewidth=.5mm](-.866,1.5)(-6.062,4.5)
\psline[linewidth=.5mm](-.866,1.5)(-2.598,1.5)
\pspolygon[fillcolor=gray,fillstyle=solid](-.866,1.5)(0,1)(-.433,.75)

\psline[linewidth=.5mm](0,3)(0,6)
\psline[linewidth=.5mm](.866,4.5)(.866,6)
\psline[linewidth=.5mm](.866,4.5)(0,3)
\pspolygon[fillcolor=gray,fillstyle=solid](0,3)(0,2)(.433,2.25)

\psline[linewidth=.5mm](2.598,-1.5)(6.062,-3.5)
\psline[linewidth=.5mm](1.732,0)(6.062,-2.5)
\psline[linewidth=.5mm](1.732,0)(2.598,-1.5)
\pspolygon[fillcolor=gray,fillstyle=solid](1.732,0)(.866,0)(.866,.5)

\psline[linewidth=.5mm](2.598,1.5)(6.062,3.5)
\psline[linewidth=.5mm](.866,1.5)(6.062,4.5)
\psline[linewidth=.5mm](.866,1.5)(2.598,1.5)
\pspolygon[fillcolor=gray,fillstyle=solid](.866,1.5)(0,1)(.433,.75)

\SpecialCoor
\psline(0,-6)(0,6)
\multido{\n=1+1}{7}{
\psline(\n;30)(\n;330)}
\rput(1;30){\psline(0,0)(0,5.5)}
\rput(1;330){\psline(0,0)(0,-5.5)}
\rput(2;30){\psline(0,0)(0,5)}
\rput(2;330){\psline(0,0)(0,-5)}
\rput(3;30){\psline(0,0)(0,4.5)}
\rput(3;330){\psline(0,0)(0,-4.5)}
\rput(4;30){\psline(0,0)(0,4)}
\rput(4;330){\psline(0,0)(0,-4)}
\rput(5;30){\psline(0,0)(0,3.5)}
\rput(5;330){\psline(0,0)(0,-3.5)}
\rput(6;30){\psline(0,0)(0,3)}
\rput(6;330){\psline(0,0)(0,-3)}
\rput(7;30){\psline(0,0)(0,2.5)}
\rput(7;330){\psline(0,0)(0,-2.5)}
\multido{\n=1+1}{7}{
\psline(\n;150)(\n;210)}
\rput(1;150){\psline(0,0)(0,5.5)}
\rput(1;210){\psline(0,0)(0,-5.5)}
\rput(2;150){\psline(0,0)(0,5)}
\rput(2;210){\psline(0,0)(0,-5)}
\rput(3;150){\psline(0,0)(0,4.5)}
\rput(3;210){\psline(0,0)(0,-4.5)}
\rput(4;150){\psline(0,0)(0,4)}
\rput(4;210){\psline(0,0)(0,-4)}
\rput(5;150){\psline(0,0)(0,3.5)}
\rput(5;210){\psline(0,0)(0,-3.5)}
\rput(6;150){\psline(0,0)(0,3)}
\rput(6;210){\psline(0,0)(0,-3)}
\rput(7;150){\psline(0,0)(0,2.5)}
\rput(7;210){\psline(0,0)(0,-2.5)}
\multido{\n=1.5+1.5}{4}{
\psline(-6.062,\n)(6.062,\n)}
\multido{\n=0+1.5}{5}{
\psline(-6.062,-\n)(6.062,-\n)}
\psline(0;0)(7;30)

\rput(0,1){\psline(0;0)(7;30)}
\rput(0,1){\psline(0;0)(7;210)}
\rput(0,2){\psline(0;0)(7;30)}
\rput(0,2){\psline(0;0)(7;210)}
\rput(0,3){\psline(0;0)(6;30)}
\rput(0,3){\psline(0;0)(7;210)}
\rput(0,4){\psline(0;0)(4;30)}
\rput(0,4){\psline(0;0)(7;210)}
\rput(0,5){\psline(0;0)(2;30)}
\rput(0,5){\psline(0;0)(7;210)}
\rput(0,6){\psline(0;0)(7;210)}
\rput(-1.732,6){\psline(0;0)(5;210)}
\rput(-3.464,6){\psline(0;0)(3;210)}
\rput(-5.196,6){\psline(0;0)(1;210)}

\rput(0,-1){\psline(0;0)(7;30)}
\rput(0,-1){\psline(0;0)(7;210)}
\rput(0,-2){\psline(0;0)(7;30)}
\rput(0,-2){\psline(0;0)(7;210)}
\rput(0,-3){\psline(0;0)(7;30)}
\rput(0,-3){\psline(0;0)(6;210)}
\rput(0,-4){\psline(0;0)(7;30)}
\rput(0,-4){\psline(0;0)(4;210)}
\rput(0,-5){\psline(0;0)(7;30)}
\rput(0,-5){\psline(0;0)(2;210)}
\rput(0,-6){\psline(0;0)(7;30)}
\rput(1.732,-6){\psline(0;0)(5;30)}
\rput(3.464,-6){\psline(0;0)(3;30)}
\rput(5.196,-6){\psline(0;0)(1;30)}

\psline(0;0)(6.928;60)
\rput(1.732,0){\psline(0;0)(6.928;60)}
\rput(3.464,0){\psline(0;0)(5.196;60)}
\rput(5.196,0){\psline(0;0)(1.732;60)}
\rput(1.732,0){\psline(0;0)(6.928;240)}
\rput(3.464,0){\psline(0;0)(6.928;240)}
\rput(5.196,0){\psline(0;0)(6.928;240)}
\rput(6.062,-1.5){\psline(0;0)(5.196;240)}
\rput(6.062,-4.5){\psline(0;0)(1.732;240)}
\rput(-1.732,0){\psline(0;0)(6.928;60)}
\rput(-3.464,0){\psline(0;0)(6.928;60)}
\rput(-5.196,0){\psline(0;0)(6.928;60)}
\rput(-1.732,0){\psline(0;0)(6.928;240)}
\rput(-3.464,0){\psline(0;0)(5.196;240)}
\rput(-5.196,0){\psline(0;0)(1.732;240)}
\rput(-6.062,1.5){\psline(0;0)(5.196;60)}
\rput(-6.062,4.5){\psline(0;0)(1.732;60)}

\psline(0;0)(6.928;120)
\psline(0;0)(6.928;300)
\rput(-1.732,0){\psline(0;0)(6.928;120)}
\rput(-3.464,0){\psline(0;0)(5.196;120)}
\rput(-5.196,0){\psline(0;0)(1.732;120)}
\rput(-1.732,0){\psline(0;0)(6.928;300)}
\rput(-3.464,0){\psline(0;0)(6.928;300)}
\rput(-5.196,0){\psline(0;0)(6.928;300)}
\rput(-6.062,-1.5){\psline(0;0)(5.196;300)}
\rput(-6.062,-4.5){\psline(0;0)(1.732;300)}
\rput(1.732,0){\psline(0;0)(6.928;300)}
\rput(3.464,0){\psline(0;0)(5.196;300)}
\rput(5.196,0){\psline(0;0)(1.732;300)}
\rput(1.732,0){\psline(0;0)(6.928;120)}
\rput(3.464,0){\psline(0;0)(6.928;120)}
\rput(5.196,0){\psline(0;0)(6.928;120)}
\rput(6.062,1.5){\psline(0;0)(5.196;120)}
\rput(6.062,4.5){\psline(0;0)(1.732;120)}

\rput(0,1){\psline(0;0)(7;150)}
\rput(0,1){\psline(0;0)(7;330)}
\rput(0,2){\psline(0;0)(7;150)}
\rput(0,2){\psline(0;0)(7;330)}
\rput(0,3){\psline(0;0)(6;150)}
\rput(0,3){\psline(0;0)(7;330)}
\rput(0,4){\psline(0;0)(4;150)}
\rput(0,4){\psline(0;0)(7;330)}
\rput(0,5){\psline(0;0)(7;330)}
\rput(0,6){\psline(0;0)(7;330)}
\rput(-1.732,6){\psline(0;0)(2;330)}

\rput(0,-1){\psline(0;0)(7;150)}
\rput(0,-1){\psline(0;0)(7;330)}
\rput(0,-2){\psline(0;0)(7;150)}
\rput(0,-2){\psline(0;0)(7;330)}
\rput(0,-3){\psline(0;0)(7;150)}
\rput(0,-3){\psline(0;0)(6;330)}
\rput(0,-4){\psline(0;0)(7;150)}
\rput(0,-4){\psline(0;0)(4;330)}
\rput(0,-5){\psline(0;0)(7;150)}
\rput(0,-5){\psline(0;0)(2;330)}
\rput(0,-6){\psline(0;0)(7;150)}

\rput(6.062,3.5){\psline(0;0)(5;150)}
\rput(6.062,4.5){\psline(0;0)(3;150)}
\rput(6.062,5.5){\psline(0;0)(1;150)}

\rput(-6.062,-3.5){\psline(0;0)(5;330)}
\rput(-6.062,-4.5){\psline(0;0)(3;330)}
\rput(-6.062,-5.5){\psline(0;0)(1;330)}

\psline(0;0)(7;330)
\psline(0;0)(7;150)
\psline(0;0)(7;210)
\psline(0;0)(6.928;240)


\end{pspicture}
\end{center}
\end{textblock}

\begin{textblock}{10.5}(10,5)
\begin{center}
\psset{unit=.5cm}
\begin{pspicture}(-6,-1)(9,9.4)
\psset{linewidth=.1mm}

\pspolygon[fillcolor=lightgray,fillstyle=solid](0,0)(0,3)(.866,1.5)

\psline[linewidth=.5mm]{->}(0,0)(.866,1.5)
\pspolygon[fillcolor=gray,fillstyle=solid](0,3)(.866,4.5)(.866,1.5)

\pspolygon[fillcolor=lightgray,fillstyle=solid](0,3)(0,6)(.866,4.5)

\pspolygon[fillcolor=gray,fillstyle=solid](0,6)(.866,4.5)(.866,7.5)

\psline(0,0)(.866,1.5)
\psline(0,0)(0,9)

\psline(0,9)(.866,9)
\psline(0,7.5)(.866,7.5)
\psline(0,6)(.866,6)
\psline(0,4.5)(.866,4.5)
\psline(0,3)(.866,3)
\psline(0,1.5)(.866,1.5)

\psline(0,1)(.866,1.5)
\psline(0,2)(.866,2.5)
\psline(0,3)(.866,3.5)
\psline(0,4)(.866,4.5)
\psline(0,5)(.866,5.5)
\psline(0,6)(.866,6.5)
\psline(0,7)(.866,7.5)
\psline(0,8)(.866,8.5)

\psline(0,1)(.433,.75)
\psline(0,2)(.866,1.5)
\psline(0,3)(.866,2.5)
\psline(0,4)(.866,3.5)
\psline(0,5)(.866,4.5)
\psline(0,6)(.866,5.5)
\psline(0,7)(.866,6.5)
\psline(0,8)(.866,7.5)
\psline(0,9)(.866,8.5)

\psline(.866,1.5)(.866,9)

\psline(0,3)(.866,1.5)
\psline(.866,4.5)(0,3)
\psline(0,6)(.866,4.5)
\psline(0.866,7.5)(0,6)
\psline(0,9)(.866,7.5)
\rput(1,.8){$t$}

\end{pspicture}\end{center}
\end{textblock}
\end{figure}
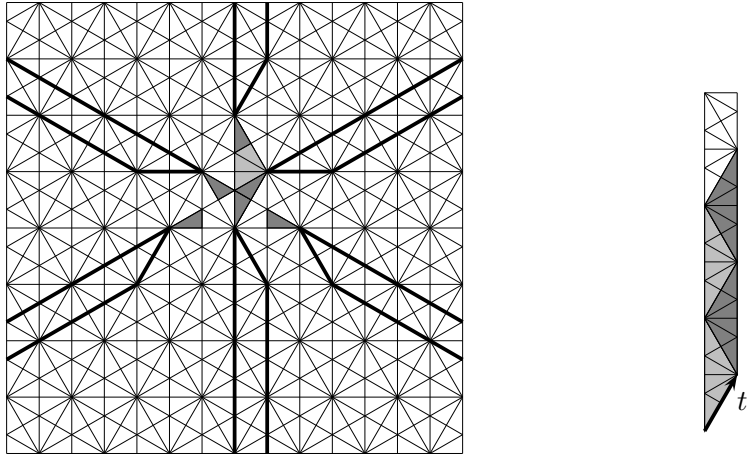

$\ $\\

\vspace{5.7cm}
We now have a look at the two-sided cell $\tc_{2}$. We set $w_{2}=s_{1}s_{3}$. In Figure \ref{c2}, we show the elements of the set $\cZ=\{z_{1},\ldots,z_{6}\}$: these are the elements which correspond to the alcoves in dark gray. We set  $t=s_{1}s_{3}s_{2}$ and 
$$\cT:=\{t^{n}|n\in \nN\}.$$
Then we have
\begin{align*}
\tc_{2}&=\{z_{i}^{-1}\tau w_{2}z_{j}\mid \tau \in \cT, 1\leq i,j\leq 6\}\\
\tc_{2}^{j}&=\{z_{i}^{-1}\tau w_{2}z_{j}\mid \tau \in \cT, 1\leq j\leq 6\}.
\end{align*}
This is explained in Figure \ref{c2}.

\begin{figure}[h!]
\caption{Description of the two-sided cell $\tc_{2}$}
\label{c2}
\begin{textblock}{10.5}(3,17)
\begin{center}
\psset{unit=.5cm}
\begin{pspicture}(-6,-7.4)(6,6.4)
\psset{linewidth=.1mm}

\psline[linewidth=.5mm](-.866,0)(-1.732,0)
\psline[linewidth=.5mm](-1.732,0)(-5.196,-6)
\psline[linewidth=.5mm](-.866,0)(-.866,-1.5)
\psline[linewidth=.5mm](-.866,-1.5)(-3.464,-6)
\pspolygon[fillcolor=gray,fillstyle=solid](0,0)(-.866,0)(-.866,0.5)

\psline[linewidth=.5mm](.866,0)(1.732,0)
\psline[linewidth=.5mm](1.732,0)(5.196,-6)
\psline[linewidth=.5mm](.866,0)(.866,-1.5)
\psline[linewidth=.5mm](.866,-1.5)(3.464,-6)
\pspolygon[fillcolor=gray,fillstyle=solid](0,0)(.866,0)(.866,0.5)

\psline[linewidth=.5mm](-.433,.75)(-.866,1.5)
\psline[linewidth=.5mm](-.433,.75)(-1.732,0)
\psline[linewidth=.5mm](-.866,1.5)(-6.062,1.5)
\psline[linewidth=.5mm](-1.732,0)(-6.062,0)
\pspolygon[fillcolor=gray,fillstyle=solid](0,0)(0,1)(.433,.75)

\psline[linewidth=.5mm](.433,.75)(.866,1.5)
\psline[linewidth=.5mm](.433,.75)(1.732,0)
\psline[linewidth=.5mm](.866,1.5)(6.062,1.5)
\psline[linewidth=.5mm](1.732,0)(6.062,0)
\pspolygon[fillcolor=gray,fillstyle=solid](0,0)(0,1)(-.433,.75)

\psline[linewidth=.5mm](0,1.5)(0,3)
\psline[linewidth=.5mm](0,1.5)(-.866,1.5)
\psline[linewidth=.5mm](-.866,1.5)(-3.464,6)
\psline[linewidth=.5mm](0,3)(-1.732,6)
\pspolygon[fillcolor=gray,fillstyle=solid](0,1.5)(.866,1.5)(0,1)

\psline[linewidth=.5mm](.433,2.25)(0,3)
\psline[linewidth=.5mm](.433,2.25)(1.732,3)
\psline[linewidth=.5mm](0,3)(1.732,6)
\psline[linewidth=.5mm](1.732,3)(3.464,6)
\pspolygon[fillcolor=gray,fillstyle=solid](0,2)(.433,2.25)(.866,1.5)

\SpecialCoor
\psline(0,-6)(0,6)
\multido{\n=1+1}{7}{
\psline(\n;30)(\n;330)}
\rput(1;30){\psline(0,0)(0,5.5)}
\rput(1;330){\psline(0,0)(0,-5.5)}
\rput(2;30){\psline(0,0)(0,5)}
\rput(2;330){\psline(0,0)(0,-5)}
\rput(3;30){\psline(0,0)(0,4.5)}
\rput(3;330){\psline(0,0)(0,-4.5)}
\rput(4;30){\psline(0,0)(0,4)}
\rput(4;330){\psline(0,0)(0,-4)}
\rput(5;30){\psline(0,0)(0,3.5)}
\rput(5;330){\psline(0,0)(0,-3.5)}
\rput(6;30){\psline(0,0)(0,3)}
\rput(6;330){\psline(0,0)(0,-3)}
\rput(7;30){\psline(0,0)(0,2.5)}
\rput(7;330){\psline(0,0)(0,-2.5)}
\multido{\n=1+1}{7}{
\psline(\n;150)(\n;210)}
\rput(1;150){\psline(0,0)(0,5.5)}
\rput(1;210){\psline(0,0)(0,-5.5)}
\rput(2;150){\psline(0,0)(0,5)}
\rput(2;210){\psline(0,0)(0,-5)}
\rput(3;150){\psline(0,0)(0,4.5)}
\rput(3;210){\psline(0,0)(0,-4.5)}
\rput(4;150){\psline(0,0)(0,4)}
\rput(4;210){\psline(0,0)(0,-4)}
\rput(5;150){\psline(0,0)(0,3.5)}
\rput(5;210){\psline(0,0)(0,-3.5)}
\rput(6;150){\psline(0,0)(0,3)}
\rput(6;210){\psline(0,0)(0,-3)}
\rput(7;150){\psline(0,0)(0,2.5)}
\rput(7;210){\psline(0,0)(0,-2.5)}
\multido{\n=1.5+1.5}{4}{
\psline(-6.062,\n)(6.062,\n)}
\multido{\n=0+1.5}{5}{
\psline(-6.062,-\n)(6.062,-\n)}
\psline(0;0)(7;30)

\rput(0,1){\psline(0;0)(7;30)}
\rput(0,1){\psline(0;0)(7;210)}
\rput(0,2){\psline(0;0)(7;30)}
\rput(0,2){\psline(0;0)(7;210)}
\rput(0,3){\psline(0;0)(6;30)}
\rput(0,3){\psline(0;0)(7;210)}
\rput(0,4){\psline(0;0)(4;30)}
\rput(0,4){\psline(0;0)(7;210)}
\rput(0,5){\psline(0;0)(2;30)}
\rput(0,5){\psline(0;0)(7;210)}
\rput(0,6){\psline(0;0)(7;210)}
\rput(-1.732,6){\psline(0;0)(5;210)}
\rput(-3.464,6){\psline(0;0)(3;210)}
\rput(-5.196,6){\psline(0;0)(1;210)}

\rput(0,-1){\psline(0;0)(7;30)}
\rput(0,-1){\psline(0;0)(7;210)}
\rput(0,-2){\psline(0;0)(7;30)}
\rput(0,-2){\psline(0;0)(7;210)}
\rput(0,-3){\psline(0;0)(7;30)}
\rput(0,-3){\psline(0;0)(6;210)}
\rput(0,-4){\psline(0;0)(7;30)}
\rput(0,-4){\psline(0;0)(4;210)}
\rput(0,-5){\psline(0;0)(7;30)}
\rput(0,-5){\psline(0;0)(2;210)}
\rput(0,-6){\psline(0;0)(7;30)}
\rput(1.732,-6){\psline(0;0)(5;30)}
\rput(3.464,-6){\psline(0;0)(3;30)}
\rput(5.196,-6){\psline(0;0)(1;30)}

\psline(0;0)(6.928;60)
\rput(1.732,0){\psline(0;0)(6.928;60)}
\rput(3.464,0){\psline(0;0)(5.196;60)}
\rput(5.196,0){\psline(0;0)(1.732;60)}
\rput(1.732,0){\psline(0;0)(6.928;240)}
\rput(3.464,0){\psline(0;0)(6.928;240)}
\rput(5.196,0){\psline(0;0)(6.928;240)}
\rput(6.062,-1.5){\psline(0;0)(5.196;240)}
\rput(6.062,-4.5){\psline(0;0)(1.732;240)}
\rput(-1.732,0){\psline(0;0)(6.928;60)}
\rput(-3.464,0){\psline(0;0)(6.928;60)}
\rput(-5.196,0){\psline(0;0)(6.928;60)}
\rput(-1.732,0){\psline(0;0)(6.928;240)}
\rput(-3.464,0){\psline(0;0)(5.196;240)}
\rput(-5.196,0){\psline(0;0)(1.732;240)}
\rput(-6.062,1.5){\psline(0;0)(5.196;60)}
\rput(-6.062,4.5){\psline(0;0)(1.732;60)}

\psline(0;0)(6.928;120)
\psline(0;0)(6.928;300)
\rput(-1.732,0){\psline(0;0)(6.928;120)}
\rput(-3.464,0){\psline(0;0)(5.196;120)}
\rput(-5.196,0){\psline(0;0)(1.732;120)}
\rput(-1.732,0){\psline(0;0)(6.928;300)}
\rput(-3.464,0){\psline(0;0)(6.928;300)}
\rput(-5.196,0){\psline(0;0)(6.928;300)}
\rput(-6.062,-1.5){\psline(0;0)(5.196;300)}
\rput(-6.062,-4.5){\psline(0;0)(1.732;300)}
\rput(1.732,0){\psline(0;0)(6.928;300)}
\rput(3.464,0){\psline(0;0)(5.196;300)}
\rput(5.196,0){\psline(0;0)(1.732;300)}
\rput(1.732,0){\psline(0;0)(6.928;120)}
\rput(3.464,0){\psline(0;0)(6.928;120)}
\rput(5.196,0){\psline(0;0)(6.928;120)}
\rput(6.062,1.5){\psline(0;0)(5.196;120)}
\rput(6.062,4.5){\psline(0;0)(1.732;120)}

\rput(0,1){\psline(0;0)(7;150)}
\rput(0,1){\psline(0;0)(7;330)}
\rput(0,2){\psline(0;0)(7;150)}
\rput(0,2){\psline(0;0)(7;330)}
\rput(0,3){\psline(0;0)(6;150)}
\rput(0,3){\psline(0;0)(7;330)}
\rput(0,4){\psline(0;0)(4;150)}
\rput(0,4){\psline(0;0)(7;330)}
\rput(0,5){\psline(0;0)(7;330)}
\rput(0,6){\psline(0;0)(7;330)}
\rput(-1.732,6){\psline(0;0)(2;330)}

\rput(0,-1){\psline(0;0)(7;150)}
\rput(0,-1){\psline(0;0)(7;330)}
\rput(0,-2){\psline(0;0)(7;150)}
\rput(0,-2){\psline(0;0)(7;330)}
\rput(0,-3){\psline(0;0)(7;150)}
\rput(0,-3){\psline(0;0)(6;330)}
\rput(0,-4){\psline(0;0)(7;150)}
\rput(0,-4){\psline(0;0)(4;330)}
\rput(0,-5){\psline(0;0)(7;150)}
\rput(0,-5){\psline(0;0)(2;330)}
\rput(0,-6){\psline(0;0)(7;150)}

\rput(6.062,3.5){\psline(0;0)(5;150)}
\rput(6.062,4.5){\psline(0;0)(3;150)}
\rput(6.062,5.5){\psline(0;0)(1;150)}

\rput(-6.062,-3.5){\psline(0;0)(5;330)}
\rput(-6.062,-4.5){\psline(0;0)(3;330)}
\rput(-6.062,-5.5){\psline(0;0)(1;330)}

\psline(0;0)(7;330)
\psline(0;0)(7;150)
\psline(0;0)(7;210)
\psline(0;0)(6.928;240)

\end{pspicture}
\end{center}
\end{textblock}

\begin{textblock}{10.5}(10,18)
\begin{center}
\psset{unit=.5cm}
\begin{pspicture}(-6,-1)(9,9.4)
\psset{linewidth=.1mm}

\pspolygon[fillcolor=lightgray,fillstyle=solid](.433,-.75)(0,0)(.866,1.5)(.866,0)(1.732,0)

\pspolygon[fillcolor=gray,fillstyle=solid](.866,0)(1.732,0)(2.598,1.5)(1.3,.75)(.866,1.5)

\pspolygon[fillcolor=lightgray,fillstyle=solid](1.3,.75)(0.866,1.5)(1.732,3)(1.732,1.5)(2.598,1.5)

\pspolygon[fillcolor=gray,fillstyle=solid](1.732,1.5)(1.732,3)(2.165,2.25)(3.464,3)(2.598,1.5)

\psline(.433,-.75)(0,0)
\psline(.433,-.75)(1.732,0)
\psline(0,0)(3.464,6)
\psline(1.732,0)(5.196,6)

\psline(0,0)(.866,-.5)
\psline(.433,.75)(1.732,0)
\psline(.866,1.5)(2.196,.75)
\psline(1.3,2.25)(2.598,1.5)
\psline(1.732,3)(3.031,2.25)
\psline(2.165,3.75)(3.464,3)
\psline(2.598,4.5)(3.897,3.75)
\psline(3.031,5.25)(4.33,4.5)
\psline(3.464,6)(4.763,5.25)

\psline(0,0)(1.732,0)
\psline(.866,1.5)(2.598,1.5)
\psline(1.732,3)(3.464,3)
\psline(2.598,4.5)(4.33,4.5)
\psline(3.464,6)(5.196,6)

\psline(0,0)(2.598,1.5)
\psline(.866,1.5)(3.464,3)
\psline(1.732,3)(4.33,4.5)
\psline(2.598,4.5)(5.296,6)

\psline(.866,1.5)(1.732,0)
\psline(1.732,3)(2.598,1.5)
\psline(2.598,4.5)(3.464,3)
\psline(3.464,6)(4.33,4.5)

\psline(.866,-.5)(.866,1.5)
\psline(1.732,0)(1.732,3)
\psline(2.598,1.5)(2.598,4.5)
\psline(3.464,3)(3.464,6)
\psline(4.33,4.5)(4.33,6)

\end{pspicture}\end{center}
\end{textblock}
\end{figure}
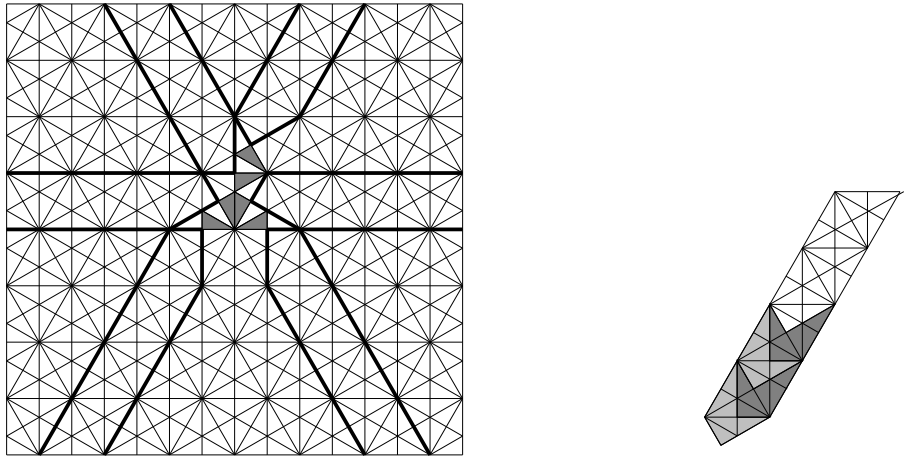

\vspace{7cm}
\noindent
We now have a look at $\tc_{3}$. We set $w_{\tc_{3}}=s_{1}$, $\cZ=\{z_{1},z_{2},z_{3}\}=\{e,s_{2},s_{2}s_{3}\}$ and
$$\cT:=\{e,s_{1}s_{2}\}.$$
Then we have
\begin{align*}
\tc_{3}&=\{z_{i}^{-1}\tau w_{1}z_{j}\mid \tau \in \cT, 1\leq i,j\leq 3\}\\
\tc_{3}^{j}&=\{z_{i}^{-1}\tau w_{1}z_{j}\mid \tau \in \cT, 1\leq j\leq 3\}.
\end{align*}

\bigskip

Finally, we have a look at $\tc_{4}=\{s_{2}s_{3}s_{2}\}$. In this case, if  we set $w_{\tc_{4}}=s_{2}s_{3}s_{2}$ and $\cZ=\cT=\{e\}$, the description in Section \ref{description} clearly holds.  

\appendix

\section{Some additional data}
\label{appendix}
The aim of this Appendix is to gather some data about cells in affine Weyl groups of rank 2 which are needed in the proof of Proposition \ref{affine-cellular-ideal}. We refer to \cite{comp} and \cite{jeju4} for details. \\

In this appendix, $(W,S)$ will denote an affine Weyl group of type $G$ or $B$ together with a generic weight function $L$. We set
$$\fC=\bigcup_{I\subsetneq S} W_{I}.$$
Let $x,y\in \fC$; we write $x\sim_{\cL\cR,\fC} y$ if there exist a sequence $x=x_{0},...,x_{n}=y$ in $\fC$ and a sequence $I_{0},...,I_{n-1}$ of subsets of $S$ such that
$$\text{$x_{k},x_{k+1}\in W_{I_{k}}$ and $x_{k}\sim_{\cLR}x_{k+1}$ in $W_{I_{k}}$}$$
for all $0\leq k\leq n-1$. This an equivalence relation and the equivalence classes will be called (for obvious reasons) the two-sided cells of $\fC$. We denote by $\cP_{\cL\cR,\fC}$ the associated partition of~$\fC$. It can be shown that the Lusztig $\ba$-function is constant on each of the equivalence classes. To each $c\in \cP_{\cL\cR,\fC}$, starting from the one with highest $\ba$-value, we associate the following subset of $W$:
$$\tc=\{w\in W\ |\ w=xuy,\ell(w)=\ell(x)+\ell(u)+\ell(y),\ x,y\in W, u\in c\}-\bigcup_{\ba(c')>\ba(c)} \tc'.$$
Then the sets $\tc$ are the two-sided cells of $W$ with respect to $L$ and the left cells lying in $\tc$ are the connected component of $\tc$.  
\begin{Rem}
In \cite{jeju4}, we introduced another equivalence relation denoted $\sim_{\fC}$. In our case, since the weight function is generic, it can be shown that the two equivalence relations $\sim_{\cL\cR,\fC}$ and $\sim_{\fC}$ are the same.
\end{Rem}
\noindent
For each choice of parameters, we give the following data:
\begin{enumerate}
\item the partition $\cP_{\cL\cR,\fC}$;
\item an ordering of $\cP_{\cL\cR,\fC}$ with respect to Lusztig $\ba$-function.
\end{enumerate}
These data determine the partition of $W$ into cells. The explicit partition can be found in \cite{jeju4} in type $G$ and in \cite{comp} in type $B$. 
\begin{Rem}
We sometime write $c_{i}\leftrightarrow c_{j}$ in the ordering of $\cP_{\cL\cR,\fC}$ to signify that for some values of the parameters we have $\ba(c_{i})>\ba(c_{j})$ and for some others we have $\ba(c_{j})\geq \ba(c_{i})$ but the corresponding sets $\tc_{i}$ and $\tc_{j}$ are the same whether $\tc_{i}$ or $\tc_{j}$ is computed first in the process.
\end{Rem}

Let $\tc$ be a two-sided cell associated to $c\in\cP_{\cL\cR,\fC}$ with left cells decomposition 
$$\tc:=\bigcup \tc^{i} \text{ ($i=1\ldots m$)}.$$
We define an element $w_{c}\in W$ and two subsets $\cZ:=\{z_{1},\ldots,z_{m}\}$ and $\cT$ such that 
$$\begin{array}{rcl}
\tc&=&\{z_{i}^{-1}\tau w_{c} z_{j}\mid \tau \in \cT, 1\leq i,j\leq m\},\\
\tc^{j}&=&\{z_{i}^{-1}\tau w_{c}z_{j}\mid \tau \in \cT, 1\leq i\leq m\},\\
(\tc^{i})^{-1}\cap \tc^{j}&=&\{z_{i}^{-1}\tau w_{c}z_{j}\mid \tau\in\cT\}.
\end{array}$$
We introduce the following notation for $t\in W$:
\begin{align*}
\sg t\sd&=\{t^{n}\mid n\in\nN\},\\
\widehat{t}&=\{w\in W\mid \ell(w^{-1}t)=\ell(t)-\ell(w)\}.
\end{align*}
\begin{Rem}
In other word the set $\hat{t}$ consists of all the elements $s_{1}\ldots s_{k}$ where $k\leq \ell(t)$ such that there exists a reduced expression of $t$ of the form $s_{1}\ldots s_{k}s_{k+1}\ldots s_{\ell(t)}$. For instance 
$$\widehat{s_{1}s_{2}s_{3}}=\{e,s_{1},s_{1}s_{2},s_{1}s_{2}s_{3}\}\text{ and }\widehat{s_{1}s_{3}}=\{e,s_{1},s_{3},s_{1}s_{3}\}$$
\end{Rem}

\subsection{Affine Weyl group of type $G$}
We keep the setting of Example \ref{G2}. As far as the lowest two-sided cell $\tc_{0}$ is concerned, the sets $\cT,\cZ$ and the element $w_{c_{0}}$ are the same as in Section \ref{examples} for all choices of parameters.

\bigskip

\noindent
{\bf Case ${\bf r>1}$.}
{\small
\begin{table}[htbp] \caption{Partition $\cP_{\cL\cR,\fC}$ and values of the $\ba$-function } 
\begin{center}
$\renewcommand{\arraystretch}{1.2}
\begin{array}{|l|c|} \hline
c_{6}=\{e\} & 0 \\
c_{5}=W_{2,3}-\{w_{2,3},e\} & b\\
c_{4}=\{w_{2,3}\}&  3b\\
c_{3}=W_{1,2}-\{e,s_{2},s_{1}s_{2}s_{1}s_{2}s_{1},w_{1,2}\}&  a\\
c_{2}=\{w_{1,3}\} & a+b \\
c_{1}=\{s_{1}s_{2}s_{1}s_{2}s_{1}\}& 3a-2b\\ 
c_{0}= \{w_{1,2}\} & 3a+3b\\ \hline
\end{array}$
\end{center}
\end{table}
}

{\small 
\begin{table}[htbp] \caption{Ordering of the partition $\cP_{\fC}$ } 
\begin{center}
\renewcommand{\arraystretch}{1.2}
$\begin{array}{|c|ccccccc|} \hline
r>2 & c_{0} & c_{1} & c_{2} & \multicolumn{2}{l}{c_{3}  \overset{}{\leftrightarrow} c_{4}} & c_{5} & c_{6}\\ \hline

2>r>3/2 &c_{0} & \multicolumn{2}{c}{c_{1}  \overset{}{ \leftrightarrow} c_{4}} & c_{2} & c_{3} & c_{5} & c_{6}\\ \hline

3/2>r>1 &c_{0} & c_{4} & c_{2} & c_{1} & c_{3} & c_{5} & c_{6}\\ \hline

 \end{array}$
\end{center}
\end{table}}

\begin{table}[h!] \caption{The sets $\cT$ and $\cZ$} 
{\small
\begin{center}
$
\renewcommand{\arraystretch}{1.4}
\begin{array}{|cc|c|c|c|c|} \hline
&& r>2 & 2>r>3/2 & 3/2>r>1&w_{c_{i}} \\\hline 
\tc_{1},& \cT & \sg s_{1}s_{2}s_{1}s_{2}s_{3}\sd &\sg s_{2}s_{2}s_{1}s_{2}s_{3}\sd &\{e\}&s_{1}s_{2}s_{1}s_{2}s_{1}   \\
 & \cZ& \widehat{s_{3}s_{2}s_{1}s_{2}s_{3}} &\widehat{s_{3}s_{2}s_{1}s_{2}s_{3}}  & \{e\} &\\\hline
\tc_{2},&\cT& \sg s_{1}s_{3}s_{2}\sd & \{e\}& \sg s_{1}s_{3}s_{2}s_{1}s_{2}\sd & s_{1}s_{3} \\
 &\cZ& \widehat{s_{2}s_{1}s_{2}s_{3}}\cup\{s_{2}s_{3}\} & \widehat{s_{2}s_{1}s_{2}s_{3}} &\widehat{s_{2}s_{1}s_{2}s_{1}}\cup\{s_{2}s_{1}s_{2}s_{3}\}  & \\\hline
\tc_{3},& \cT & \{e,s_{1}s_{2}\} & \{e,s_{1}s_{2}\}&\{e,s_{1}s_{2}\} & s_{1} \\
 &\cZ& \widehat{s_{2}s_{3}} &  \widehat{s_{2}s_{3}} &  \widehat{s_{2}s_{3}}&  \\\hline
\tc_{4},&\cT & \{e\} & \sg s_{3}s_{2}s_{1}\sd& \sg s_{3}s_{2}s_{1}\sd  & s_{2}s_{3}s_{2} \\
 &\cZ& \{e\} &\widehat{s_{1}s_{2}s_{1}s_{2}s_{3}} & \widehat{s_{1}s_{2}s_{1}s_{2}s_{3}}&  \\\hline

\end{array}
$
\end{center}}
\end{table}
\newpage
\vspace{1cm}
\noindent
{\bf Case ${\bf r<1}$.}

{\small
\begin{table}[htbp] \caption{Partition $\cP_{\cL\cR,\fC}$ and value of the $\ba$-function when $b>a$} 
\begin{center}
$\renewcommand{\arraystretch}{1.3}
\begin{array}{|l|c|} \hline
c_{6}=\{e\} & 0 \\
c_{5}=\{s_{1}\} & a\\
c_{4}=\fC -\{e,s_{1},s_{2}s_{1}s_{2}s_{1}s_{2},w_{1,2},w_{1,3},w_{2,3}\}&  b\\
c_{3}=  \{w_{1,3}\}& a+b\\
c_{2}=\{s_{2}s_{1}s_{2}s_{1}s_{2}\} & 3b-2a \\
c_{1}=\{w_{2,3}\}& 3b\\ 
c_{0}= \{w_{1,2}\} & 3a+3b\\ \hline
\end{array}$
\end{center}
\end{table}
}

We get the following ordering 
$$c_{0}\ c_{1}\ c_{2}\leftrightarrow c_{3} \ c_{4} \ c_{5} \ c_{6}.$$

\begin{table}[h!] \caption{The sets $\cT$ and $\cZ$} 
{\footnotesize
\begin{center}
$
\renewcommand{\arraystretch}{1.4}
\begin{array}{|cc|c|c|c|} \hline
&& r<1 & w_{c_{i}}\\\hline 
\tc_{1},& \cT & \sg s_{3}s_{2}s_{1}\sd & s_{2}s_{3}s_{2}   \\
 & \cZ& \widehat{s_{1}s_{2}s_{1}s_{2}s_{3}} &\\\hline
 \tc_{2},& \cT & \{e\}   & s_{2}s_{1}s_{2}s_{1}s_{2} \\
 & \cZ&\widehat{s_{3}} &\\\hline
 \tc_{3},& \cT & \sg s_{1}s_{3}s_{2}s_{1}s_{2}\sd & s_{1}s_{3}   \\
 & \cZ&\widehat{s_{2}s_{1}s_{2}s_{1}}\cup\{s_{2}s_{1}s_{2}s_{3}\}& \\\hline
  \tc_{5},& \cT & \{e\}  &s_{1}  \\
 & \cZ&\{e\}&\\\hline
 \end{array}
$
\end{center}}
\end{table}
$\ $\\

\subsection{Affine Weyl group of type $B$}
We keep the setting of Example \ref{B2}. As far as the lowest two-sided cell $\tc_{0}$ is concerned, the set $\cT$, $\cZ$ and the element $w_{c_{0}}$ are the same for all choices of parameters, namely $w_{c_{0}}=s_{1}s_{2}s_{1}s_{2}$, 
$$\cZ=\widehat{s_{3}s_{2}s_{1}s_{3}s_{2}s_{3}}$$
and
$$\cT=\{t_{1}^{m}t_{2}^{n}\mid n,m\in\nN\}$$
where $t_{1}=s_{2}s_{1}s_{2}s_{3}$, $t_{2}=s_{1}s_{2}s_{1}s_{3}s_{2}s_{3}$.

\bigskip

{\bf Generic parameters in zone $A_{i}$ ($1\leq i\leq 5$).}\\

{\small
\begin{table}[h!] \caption{Partition $\cP_{LR,\fC}$ and values of the $\ba$-function} 
\begin{center}
$\renewcommand{\arraystretch}{1.4}
\begin{array}{|l|c|} \hline
c_{8}=\{e\} & 0 \\
c_{7}=\{s_{3}\} & c\\
c_{6}=\{s_{2},s_{3}s_{2},s_{2}s_{3},s_{3}s_{2}s_{3}\}& b\\ 
c_{5}=\{s_{2}s_{3}s_{2}\} & 2b-c\\
c_{4}=\{s_{2}s_{3}s_{2}s_{3}\}& 2b+2c\\ 
c_{3}=\{s_{1},s_{2}s_{1},s_{1}s_{2},s_{2}s_{1}s_{2}\}& a \\
c_{2}=\{s_{1}s_{3}\}&  a+c\\
c_{1}= \{s_{1}s_{2}s_{1}\} & 2a-b\\ 
c_{0}= \{w_{1,2}\} & 2a+2b\\ \hline
\end{array}$
\end{center}
\end{table}
}

{\small
\begin{table}[h] \caption{Partition $\cP_{\fC}$} 
\begin{center}
\renewcommand{\arraystretch}{1.4}
$\begin{array}{|c|ccccccccc|} \hline
(r_{1},r_{2})\in A_{1}  & c_{0} & c_{1} & c_{2} &\multicolumn{2}{r}{ c_{3}  \leftrightarrow c_{4}} & c_{5} & c_{6} & c_{7} & c_{8}\\ \hline
(r_{1},r_{2})\in A_{2}  &c_{0} & c_{1} & c_{4} & c_{2} & c_{3} & c_{5} & c_{6} & c_{7} & c_{8}\\ \hline 
(r_{1},r_{2})\in A_{3} & c_{0} & \multicolumn{2}{r}{  c_{1} \leftrightarrow c_{4}}  & c_{2}  & c_{5} & c_{3} & c_{6} & c_{7} & c_{8} \\ \hline
(r_{1},r_{2})\in A_{4} & c_{0} & c_{4} & c_{2} & c_{1}  & c_{5} & c_{3} & c_{6} & c_{7} & c_{8} \\ \hline
(r_{1},r_{2})\in A_{5} & c_{0} & c_{4} & c_{2} & c_{1} & c_{3} & c_{5} & c_{6} & c_{7} & c_{8} \\ \hline
\end{array}$
\end{center}
\end{table}}

\begin{table}[h!] \caption{The sets $\cZ$ and $\cT$} 
{\small
\begin{center}
$
\renewcommand{\arraystretch}{1.4}
\begin{array}{|cc|c|c|c|c|c|c|} \hline
&& A_{1} & A_{2} & A_{3} & A_{4} & A_{5}& w_{c_{i}} \\\hline 
\tc_{1},& \cT & \sg s_{1}s_{2}s_{3}\sd &\sg s_{1}s_{2}s_{3}\sd &\sg s_{1}s_{2}s_{3}\sd  &\{e\} &\{e\}&s_{1}s_{2}s_{1} \\
 & \cZ& \widehat{s_{3}s_{2}s_{3}} &\widehat{s_{3}s_{2}s_{3}}  &\widehat{s_{3}s_{2}s_{3}}& \{e\}&\{e\}& \\\hline
\tc_{2},& \cT & \sg s_{1}s_{2}s_{3}s_{2}\sd& \{e\}& \{e\} & \sg s_{1}s_{3}s_{2}\sd &\sg s_{1}s_{3}s_{2}\sd& s_{1}s_{3}\\
 & \cZ&  \widehat{s_{2}s_{3}s_{2}}&  \widehat{s_{2}s_{3}}& \widehat{s_{2}s_{3}}&\widehat{s_{2}s_{1}}\cup \{s_{2}s_{3}\}&\widehat{s_{2}s_{1}}\cup \{s_{2}s_{3}\}& \\\hline
\tc_{3},& \cT & \sg s_{1}s_{2}s_{3}s_{2}\sd& \sg s_{1}s_{2}s_{3}s_{2}\sd& \{e\} & \{e\}& \sg s_{1}s_{2}s_{3}s_{2}\sd &s_{1}\\
 & \cZ&  \widehat{s_{2}s_{3}s_{2}}&   \widehat{s_{2}s_{3}s_{2}}& \widehat{s_{2}s_{3}}& \widehat{s_{2}s_{3}}&\widehat{s_{2}s_{3}s_{2}}& \\\hline
  \tc_{4},& \cT &\{e\} &\sg s_{2}s_{3}s_{2}s_{1}\sd &  \sg s_{2}s_{3}s_{2}s_{1}\sd & \sg s_{2}s_{3}s_{2}s_{1}\sd &\sg s_{2}s_{3}s_{2}s_{1}\sd &s_{2}s_{3}s_{2}s_{3}\\
 & \cZ& \{e\} &\widehat{s_{1}s_{2}s_{3}}  &\widehat{s_{1}s_{2}s_{3}}&\widehat{s_{1}s_{2}s_{3}} &\widehat{s_{1}s_{2}s_{3}}& \\\hline
 \tc_{5},& \cT &\{e\} &\{e\} & \sg s_{2}s_{3}s_{2}s_{1}\sd &\sg s_{2}s_{3}s_{2}s_{1}\sd &\{e\}&s_{2}s_{3}s_{2}\\
 & \cZ& \{e\} &\{e\}  &\widehat{s_{1}s_{2}s_{3}}&\widehat{s_{1}s_{2}s_{3}} &\{e\} &\\\hline
 \tc_{6},& \cT & \{e\}& \{e\}&\{e\}  &\{e\} &\{e\}&s_{2}\\
 & \cZ& \widehat{s_{3}} & \widehat{s_{3}} &\widehat{s_{3}}&\widehat{s_{3}} & \widehat{s_{3}}&\\\hline
 \tc_{7},& \cT &  \{e\}& \{e\} & \{e\}  & \{e\} & \{e\}&s_{3}\\
 & \cZ&  \{e\} & \{e\}  & \{e\}&  \{e\}&  \{e\}&\\\hline

\end{array}
$
\end{center}}
\end{table}
$\ $\\

\newpage

$\ $\\
\noindent
{\bf Generic parameters in zone $B_{i}$ ($i=1,2$).}\\

{\small
\begin{table}[h] \caption{Partition $\cP_{LR,\fC}$ and values of the $\ba$-function} 
\begin{center}
$\renewcommand{\arraystretch}{1.4}
\begin{array}{|l|c|} \hline
c_{8}=\{e\} & 0 \\
c_{7}=\{s_{3}\} & c\\
c_{6}=\{s_{1}\}& a\\ 
c_{5}=\{s_{1}s_{3}\}&  a+c\\
c_{4}=\{s_{2},s_{1}s_{2},s_{2}s_{1},s_{1}s_{2}s_{1},s_{3}s_{2},s_{2}s_{3},s_{3}s_{2}s_{3}\} & b\\
c_{3}=\{s_{2}s_{1}s_{2}\}& 2b-a\\ 
c_{2}=\{s_{2}s_{3}s_{2}\} & 2b-c \\
c_{1}= \{w_{2,3}\} & 2b+2c\\ 
c_{0}= \{w_{1,2}\} & 2a+2b\\ \hline
\end{array}$
\end{center}
\end{table}
}

\vspace{-.3cm}
{\small
\begin{table}[h] \caption{Partition $\cP_{\fC}$} 
\begin{center}
\renewcommand{\arraystretch}{1.4}
$\begin{array}{|c|ccccccccc|} \hline
(r_{1},r_{2})\in B_{2} & c_{0} & c_{1} & c_{2} & c_{3}  & c_{4} & c_{5} & c_{6} & c_{7} & c_{8}\\ \hline

(r_{1},r_{2})\in B_{1}& c_{0} & c_{1} & c_{2} &\multicolumn{2}{l}{c_{3}  \leftrightarrow c_{5}} & c_{4} & c_{6} & c_{7} & c_{8}\\ \hline \end{array}$
\end{center}
\end{table}}

\begin{table}[h!] \caption{The sets $\cT$ and $\cZ$} 
{\small
\begin{center}
$
\renewcommand{\arraystretch}{1.4}
\begin{array}{|cc|c|c|c|c|c|} \hline
&& B_{1} & B_{2} &  w_{c} \\\hline 
\tc_{1},& \cT &  \sg s_{2}s_{3}s_{2}s_{1}\sd  &  \sg s_{2}s_{3}s_{2}s_{1}\sd  &   s_{2}s_{3}s_{2}s_{3}\\
& \cZ&\widehat{s_{1}s_{2}s_{3}}   &  \widehat{s_{1}s_{2}s_{3}} &  \\\hline
\tc_{2},& \cT   &\sg s_{2}s_{3}s_{2}s_{1}\sd  &\sg s_{2}s_{3}s_{2}s_{1}\sd  & s_{2}s_{3}s_{2}\\
& \cZ&  \widehat{s_{1}s_{2}s_{3}} &  \widehat{s_{1}s_{2}s_{3}}   &\\\hline
\tc_{3},& \cT   & \{e\} &\{e\}  & s_{2}s_{1}s_{2}\\
& \cZ&  \hat{s_{3}} & \hat{s_{3}}    &\\\hline
\tc_{4},& \cT   & \{e\} &\sg s_{2}s_{1}s_{3}\sd  & s_{2}\\
& \cZ& \{e,s_{1},s_{3}\} &  \widehat{s_{1}s_{3}}   &\\\hline
\tc_{5},& \cT   &  \sg s_{1}s_{3}s_{2}\sd &\{e\}  & s_{1}s_{3}\\
& \cZ&  \widehat{s_{2}s_{1}}\cup \{s_{2}s_{3}\}&   \{e\}  &\\\hline   
\tc_{6},& \cT   & \{e\} &\{e\}  & s_{1}\\
& \cZ&  \{e\} & \{e\}    &\\\hline
\tc_{7},& \cT   &\{e\}  & \{e\} & s_{3}\\
& \cZ& \{e\}  & \{e\}  &\\\hline

\end{array}
$
\end{center}}
\end{table}

\newpage

$\ $\\
\noindent
{\bf Generic parameters in zone $C_{i}$ ($i=1,2,3$).}\\

{\small
\begin{table}[h] \caption{Partition $\cP_{LR,\fC}$ and values of the $\ba$-function} 
\begin{center}
\renewcommand{\arraystretch}{1.4}
$\begin{array}{|l|c|} \hline
c_{8}=\{e\} & 0\\
c_{7}=\{s_{2}\} & b\\
c_{6}=\{s_{3},s_{2}s_{3},s_{3}s_{2},s_{2}s_{3}s_{2}\}   & c\\
c_{5}=\{s_{3}s_{2}s_{3}\} & 2c-b\\
c_{4}=\{s_{2}s_{3}s_{2}s_{3}\} & 2b+2c\\
c_{3}=\{s_{1},s_{2}s_{1},s_{1}s_{2},s_{2}s_{1}s_{2}\} & a\\
c_{2}=\{s_{1}s_{3}\}& a+c\\
c_{1}=\{s_{1}s_{2}s_{1}\} & 2a-b\\
c_{0}=\{w_{1,2}\}&  2a+2b\\ \hline
 \end{array}$
\end{center}
\end{table}}

{\small
\begin{table}[h] \caption{Partition $\cP_{\fC}$} 
\begin{center}
\renewcommand{\arraystretch}{1.4}
$\begin{array}{|c|ccccccccc|} \hline
(r_{1},r_{2})\in C_{1}  &c_{0} &  c_{4}  & c_{2} & c_{1} & c_{3} & c_{5} & c_{6} & c_{7} & c_{8}\\ \hline 
(r_{1},r_{2})\in C_{2}  &c_{0} & \multicolumn{2}{r}{ c_{1}  \leftrightarrow c_{4}}& c_{2} & c_{3} & c_{5} & c_{6} & c_{7} & c_{8}\\ \hline 
(r_{1},r_{2})\in C_{3}  & c_{0} & c_{1} & c_{2} &\multicolumn{2}{r}{ c_{3}  \leftrightarrow c_{4}} & c_{5} & c_{6} & c_{7} & c_{8}\\ \hline
\end{array}$
\end{center}
\end{table}}

\begin{table}[h!] \caption{The sets $\cT$ and $\cZ$} 
{\small
\begin{center}
$
\renewcommand{\arraystretch}{1.4}
\begin{array}{|cc|c|c|c|c|c|} \hline
&& C_{1} & C_{2} & C_{3}& w_{c} \\\hline 
\tc_{1},& \cT & \{e\} &  \sg s_{1}s_{2}s_{3}\sd & \sg s_{1}s_{2}s_{3}\sd & s_{1}s_{2}s_{1}\\
& \cZ& \{e\}& \widehat{s_{3}s_{2}s_{3}} & \widehat{s_{3}s_{2}s_{3}} &\\\hline
\tc_{2},& \cT & \sg s_{1}s_{3}s_{2}\sd &\{e\} & \sg s_{1}s_{2}s_{3}s_{2}\sd& s_{1}s_{3}\\
& \cZ& \widehat{s_{2}s_{1}}\cup \{s_{2}s_{3}\}& \widehat{s_{2}s_{3}}&  \widehat{s_{2}s_{3}s_{2}}&\\\hline
\tc_{3},& \cT &  \sg s_{1}s_{2}s_{3}s_{2}\sd&\sg s_{1}s_{2}s_{3}s_{2}\sd &\sg s_{1}s_{2}s_{3}s_{2}\sd &s_{1}\\
& \cZ&   \widehat{s_{2}s_{3}s_{2}}&  \widehat{s_{2}s_{3}s_{2}} &  \widehat{s_{2}s_{3}s_{2}} &\\\hline
\tc_{4},& \cT &  \sg s_{2}s_{3}s_{2}s_{1}\sd & \sg s_{2}s_{3}s_{2}s_{1}\sd &\{e\} &s_{2}s_{3}s_{2}s_{3}\\
& \cZ& \widehat{s_{1}s_{2}s_{3}}&\widehat{s_{1}s_{2}s_{3}} &\{e\} &\\\hline
\tc_{5},& \cT &  \{e\}  & \{e\}  & \{e\}  &s_{3}s_{2}s_{3} \\
& \cZ& \{e\}  & \{e\}  & \{e\}  &\\\hline
\tc_{6},& \cT &  \{e\} &  \{e\}&  \{e\}&s_{3}\\
& \cZ&\hat{s_{2}} &\hat{s_{2}}  & \hat{s_{2}} &\\\hline
\tc_{7},& \cT & \{e\} & \{e\} &  \{e\}&s_{2}\\
& \cZ&  \{e\}& \{e\} & \{e\} &\\\hline
\end{array}
$
\end{center}}
\end{table}

\bibliographystyle{plain}

\end{document}